\documentclass[11pt]{article}

\usepackage[T1]{fontenc}
\usepackage[utf8]{inputenc}
\usepackage[english]{babel}

\usepackage[dvipsnames,svgnames]{xcolor}

\usepackage{amsmath,amssymb,amsfonts,amsthm,upgreek,verbatim}
\usepackage{stmaryrd}
\usepackage{mathtools}
\let\epsilon\varepsilon
\let\phi\varphi
\let\emptyset\varnothing
\usepackage{scalerel}
\usepackage{mleftright}
\mleftright
\usepackage{seqsplit}

\usepackage{bbm}
\usepackage{bm}

\DeclareMathAlphabet\mbc{OMS}{cmsy}{b}{n}
\DeclareMathAlphabet\mathbfcal{OMS}{cmsy}{b}{n}

\usepackage[overload]{empheq}

\usepackage{booktabs}

\usepackage{enumerate,enumitem}

\usepackage{subcaption}
\usepackage[font={small}]{caption}

\usepackage{algorithm,algpseudocode}
\newcommand{\algcomment}[1]{\hfill\textnormal{\textcolor{gray}{(#1)}}}

\usepackage{fancyhdr}
\usepackage[a4paper]{geometry}
\usepackage{url}
\usepackage{hyperref}
\hypersetup{
    colorlinks = true,
    citecolor = BrickRed,
    urlcolor = blue
}
\usepackage[capitalise,noabbrev,nameinlink]{cleveref}
\usepackage{thmtools}
\newcommand{\crefpart}[2]{%
  \hyperref[#2]{\namecref{#1}~\labelcref*{#1}~\ref*{#2}}%
}
\crefname{equation}{}{}
\crefname{enumi}{}{}
\let\eqref\cref

\theoremstyle{definition}
\newtheorem{theorem}{Theorem}[section]
\newtheorem{definition}[theorem]{Definition}
\newtheorem{proposition}[theorem]{Proposition}
\newtheorem{corollary}[theorem]{Corollary}

\newtheorem{lemma}[theorem]{Lemma}
\newtheorem{remark}[theorem]{Remark}
\newtheorem{example}[theorem]{Example}
\newtheorem{assumption}[theorem]{Assumption}
\newtheorem{fact}[theorem]{Fact}

\crefname{theorem}{Theorem}{Theorems}
\crefname{definition}{Definition}{Definitions}
\crefname{proposition}{Proposition}{Propositions}
\crefname{corollary}{Corollary}{Corollaries}
\crefname{lemma}{Lemma}{Lemmas}
\crefname{remark}{Remark}{Remarks}
\crefname{example}{Example}{Examples}
\crefname{assumption}{Assumption}{Assumptions}
\crefname{fact}{Fact}{Facts}
\crefname{conjecture}{Conjecture}{Conjectures}
\crefname{algorithm}{Algorithm}{Algorithms}

\usepackage{tcolorbox}
\tcbuselibrary{breakable}

\usepackage{tikz}
\usepackage{tikz-cd}
\usetikzlibrary{decorations.markings,arrows}
\usetikzlibrary{shapes}
\usepackage{adjustbox}
\usepackage{pgf}
\usepackage{pgfplots}
\pgfplotsset{compat=1.18}

\DeclareMathOperator{\ri}{ri}
\DeclareMathOperator{\sri}{sri}
\DeclareMathOperator{\gra}{gra}

\DeclareMathOperator{\ran}{ran}

\DeclareMathOperator{\rk}{rank}
\DeclareMathOperator{\dom}{dom}
\DeclareMathOperator{\Id}{Id}

\DeclareMathOperator{\Diag}{diag}
\DeclareMathOperator{\Lap}{Lap}

\DeclareMathOperator{\Inc}{Inc}
\DeclareMathOperator*{\argmin}{argmin}

\newcommand{\wto}{\rightharpoonup}
\newcommand{\Hi}{\mathcal{H}}

\newcommand{\Gi}{\mathcal{G}}
\newcommand{\bHi}{\mbc{H}}
\newcommand{\R}{\mathbb{R}}

\newcommand{\tto}{\rightrightarrows}

\usepackage[doi=true,isbn=false,url=false,eprint=false,sortcites=true,maxnames=5,giveninits=true,sorting=nyt,backend=biber]{biblatex}
\usepackage{csquotes}
\usepackage{xurl}
\DeclareNameAlias{author}{given-family}
\newtoggle{isjournalshort}
\settoggle{isjournalshort}{true}

\DeclareFieldFormat{doi}{%
  \urlstyle{tt}%
  \href{https://doi.org/#1}{\nolinkurl{DOI:#1}}%
}

\DeclareFieldFormat[inbook,incollection,inproceedings]{title}{#1}

\DeclareBibliographyDriver{article}{%
\printnames{author}\addperiod\space
\thefield{title}\addperiod\space
\textit{\iftoggle{isjournalshort}{\thefield{shortjournal}}{\thefield{journaltitle}}}\addperiod\addspace%
\iffieldundef{volume}{}{%
\textbf{\thefield{volume}}(\thefield{number})\addcomma\space}%
\thefield{pages}\addspace
(\thefield{year})\addperiod\space
\printfield{doi}
\finentry
}

\DeclareBibliographyDriver{misc}{%
\printnames{author}\addperiod\space
\thefield{title}\addperiod\space
(\printfield{year})\addperiod\space
\iffieldundef{eprint}{}{%
  \texttt{\href{https://arxiv.org/abs/\thefield{eprint}}{arxiv:\thefield{eprint}}}
  \addspace}%
\finentry
}

\AtEveryBibitem{%
  \ifentrytype{book}{%
    \clearfield{doi}%
  }{}%
}

\makeatletter

\newcommand{\funding}[1]{\def\@funding{#1}}
\newcommand{\interests}[1]{\def\@interests{#1}}
\newcommand{\data}[1]{\def\@data{#1}}
\newcommand{\referees}[1]{\def\@referees{#1}}

\def\@funding{}
\def\@interests{}
\def\@data{}
\def\@referees{}

\newcommand{\makeacknowledgements}{
    \ifx\@funding\@empty\else\paragraph{Acknowledgements}\@funding\fi
    \ifx\@interests\@empty\else\paragraph{Conflict of interest}\@interests\fi
    \ifx\@data\@empty\else\paragraph{Data Availability}\@data\fi
    \ifx\@referees\@empty\else\paragraph{Referees}\@referees\fi
}

\newcommand{\keywords}[1]{\def\@keywords{#1}}
\newcommand{\msc}[1]{\def\@msc{#1}}

\def\@keywords{}
\def\@msc{}

\usepackage{listofitems}
\newcommand{\dotlist}[1]{%
  \readlist*\myitems{#1}%
  \foreachitem\x\in\myitems{%
    \ifnum\xcnt=1\else~·~\fi\x%
  }%
}

\newcommand{\makepaperdata}{
  \ifx\@keywords\@empty\else\paragraph{Keywords}\dotlist{\@keywords}\fi
  \ifx\@msc\@empty\else\paragraph{MSC}\dotlist{\@msc}\fi
}

\makeatother

\newcommand*\samethanks[1][\value{footnote}]{\footnotemark[#1]}

\hypersetup{linkcolor=DarkOrchid}
\title{A unified family of multiblock ADMM schemes\\
 by means of graph splitting}
\author{
	Rub\'en Campoy\thanks{Department of Mathematics, University of Alicante, \textsc{Spain}, \href{mailto:ruben.campoy@ua.es}{ruben.campoy@ua.es}, \href{mailto:cesar.lopez@ua.es}{cesar.lopez@ua.es}.}
	\and Scott B. Lindstrom\thanks{Centre for Optimisation and Decision Science, Curtin University, \textsc{Australia}, \href{mailto:scott.lindstrom@curtin.edu.au}{scott.lindstrom@curtin.edu.au}.}
	\and C\'esar L\'opez-Pastor\samethanks[1]
}
\date{}


\msc{47H05 $\cdot$ 47N10 $\cdot$ 47J25 $\cdot$ 49M29 $\cdot$ 90C25}

\funding{R.C. and C.L.P. were partially supported by grant PID2025-168246NB-I00 funded by  MICIU/AEI/10.13039/501100011033 and by ERDF/EU, as well as by grant CIGE/2026/80 funded by Conselleria d'Educaci\'o, Cultura i Universitats, Generalitat Valenciana. C.L.P. was additionally supported by grant PREP2022-000118 funded by \seqsplit{MICIU/AEI/10.13039/501100011033} and by ``ESF Investing in your future''. S.B.L. acknowledges the Banff International Research Station Workshop \textit{Strategies for Handling Applications with Nonconvexity (25w5372)}, which seeded his collaboration in this research.}
\interests{The authors declare that they have no conflict of interest.}
\data{}
\referees{}

\begin{document}

\maketitle

\begin{abstract}
This paper introduces a novel unified family of multiblock ADMM schemes by combining Gabay duality with a graph-based splitting methodology. This framework allows us to recover existing ADMM algorithms as particular instances, as well as to generate new ones by varying the underlying network topology. We focus specifically on the case of the complete graph, a structure of significant practical interest due to its dense connectivity, for which we provide a fully explicit and directly implementable scheme. We further derive an equivalent reformulation of the proposed family in which the algorithm is expressed in terms of a reduced set of dual variables, which clarifies its connection to the classical ADMM framework. Numerical experiments are conducted to compare the performance of different graph choices, showing particularly favorable results for the complete graph.
\end{abstract}

\makepaperdata

\section{Introduction}

Given real Hilbert spaces $\Hi_1, \dots, \Hi_n$ and $\Gi$, and given $b \in \Gi$, we are interested in solving systems of coupled inclusions of the form 

\begin{equation}\tag{P}\label{primal}
\text{find } (x_1,\ldots,x_n, y) \in \Hi_1 \times \dots \times \Hi_n \times \mathcal{G} \;\text{ such that}\;
\begin{cases}
0 \in F_1(x_1) + A_1^* y, \\
\quad \vdots \\
0 \in F_n(x_n) + A_n^* y, \\
A_1 x_1 + \dots + A_n x_n = b;
\end{cases}
\end{equation}
where $F_i: \Hi_i \rightrightarrows \Hi_i$ is a maximally monotone operator and $A_i \in \mathcal{L}(\Hi_i, \mathcal{G})$ is a bounded linear operator with adjoint $A_i^*$, for each $i=1,\ldots,n$. The study of coupled systems of this form is of significant interest as it encompasses a wide range of problems in applied mathematics and engineering. In particular, when each operator $F_i$ corresponds to the subdifferential of a proper, convex, and lower semicontinuous function $f_i:\Hi_i\to{]-\infty,+\infty]}$, the inclusions in \cref{primal} become the KKT system associated with the optimization problem
\[ 	\min\; \sum_{i=1}^nf_i(x_i)\quad\text{s.t. }\sum_{i=1}^nA_ix_i=b.\]

For the case $n=2$, a well-known approach to solve the problem above is the \emph{Alternating Direction Method of Multipliers (ADMM)}. Originally introduced in the mid-1970s \cite{gabay1976dual,glowinski1975approximation}, this algorithm is designed to exploit the separable structure of the problem by updating the variables sequentially. Given a penalty parameter $\rho>0$, the iteration of ADMM is given by

\begin{equation}\label{eq:clasADMM}
    \left\{
    \begin{aligned}
    x_1^{k+1} &\in \argmin_{x_1 \in \Hi_1} \left\{ f_1(x_1) + \langle y^k, A_1 x_1 \rangle + \frac{\rho}{2} \| A_1 x_1 + A_2 x_2^k - b \|^2 \right\}, \\
    x_2^{k+1} &\in \argmin_{x_2 \in \Hi_2} \left\{ f_2(x_2) + \langle y^k, A_2 x_2 \rangle + \frac{\rho}  {2} \| A_1 x_1^{k+1} + A_2 x_2 - b \|^2 \right\}, \\
    y^{k+1} &= y^k + \rho (A_1 x_1^{k+1} + A_2 x_2^{k+1} - b).
    \end{aligned}
    \right.
\end{equation}
This scheme can be understood as a splitting strategy that performs a separate minimization of the \emph{augmented Lagrangian} associated with the problem (see~\cref{def:lagrangian}) for each variable block. 

By now, ADMM has become a fundamental tool in several fields, including large-scale optimization, signal processing, and statistical learning \cite{Boyd2011}. However, while the algorithm is well understood for two blocks, many modern applications naturally lead to the multiblock case where $n \geq 3$.  The most direct extension of the classical ADMM to this setting is a Gauss--Seidel-type implementation, where the $n$ primal block components $x_i^{k+1}$ are updated one by one before the dual update $y^{k+1}$:

\begin{equation}
    \left\{
    \begin{aligned}
        x_1^{k+1} &\in \argmin_{x_1 \in \Hi_1} \left\{ f_1(x_1) + \langle y^k, A_1 x_1 \rangle + \frac{\rho}{2} \left\| A_1 x_1 +\textstyle \sum_{j=2}^n A_j x_j^k - b \right\|^2 \right\}, \\
        & \vdots\\
        x_i^{k+1} &\in \argmin_{x_i \in \Hi_i} \left\{ f_i(x_i) + \langle y^k, A_i x_i \rangle + \frac{\rho}{2} \left\| \textstyle\sum_{j=1}^{i-1} A_j x_j^{k+1} + A_i x_i + \sum_{j=i+1}^n A_j x_j^k - b \right\|^2 \right\}, \\
        & \vdots\\
        x_n^{k+1} &\in \argmin_{x_n \in \Hi_n} \left\{ f_n(x_n) + \langle y^k, A_n x_n \rangle + \frac{\rho}{2} \left\|\textstyle \sum_{j=1}^{n-1} A_j x_j^{k+1} + A_n x_n - b \right\|^2 \right\}, \\
        y^{k+1} &= y^k + \rho \left( \textstyle\sum_{j=1}^n A_j x_j^{k+1} - b \right).
    \end{aligned}
    \right.
\end{equation}

Although this multiblock scheme may perform well in practice, it is well-known that it does not necessarily converge, even for simple quadratic problems \cite{chen2016direct}. To ensure its convergence, several strategies have been proposed in the literature. Initial efforts focused on imposing structural assumptions such as orthogonality conditions or full-rank on the matrices, or strong convexity of some terms in the objective function \cite{chen2016direct, han2012note, li2015convergent, lin2018global,hong2017linear}. Other approaches introduced algorithmic modifications such as Gaussian back-substitution \cite{he2012alternating}. Parallel strategies have also been explored, combining the iterative process with Jacobi-type updates and proximal penalization terms~\cite{deng2017parallel, he2015splitting}. For a comprehensive and recent review of these developments, the reader is referred to the survey \cite{han2022survey}.

A fundamental approach to understanding ADMM is through its duality-like relationship with the \emph{Douglas--Rachford (DR)} splitting algorithm~\cite{DR56, LM79}. This connection was first described by Gabay~\cite{gabay1983chapter}, has been further elucidated in popular accessible works by Eckstein, Yao, and others \cite{EcksteinYao2012,eckstein2015understanding,lindstrom2021survey}, and has been called \textit{Gabay's duality}~\cite{gabayduality}. This relationship permits the derivation of structural properties and convergence results for ADMM through the analysis of DR applied to the dual problem~\cite{HeYuan2012,HeYuan2015}. It also allows for principled approaches to algorithmic extrapolation based on the DR algorithm's governing sequence, whose dynamics have been frequently observed to be better suited to such schemes than the sequences by which ADMM is conventionally expressed~\cite{calcan2024admm,lindstrom2022computable,poon2019trajectory}. A key tool in this type of duality theory is Moreau's identity (see~\cref{rem:Moreau}). Recently, a \emph{complete Moreau's identity} has been introduced that extends this mechanism to fully general resolvent mappings and enables Gabay-like duality relationships to be derived without convexity or monotonicity assumptions~\cite{gabayduality}.

Analogously to the classical Douglas--Rachford setting, some convergent multiblock ADMM schemes have been derived as dual counterparts of different multi-operator extensions of the DR algorithm. For instance, alongside the introduction of the \emph{Davis--Yin}~\cite{davisyin} and \emph{Malitsky--Tam}~\cite{malitsky2023resolvent} splitting algorithms, their authors explicitly derived the resulting ADMM\footnote{Here, by `ADMM', we mean a broader class of alternating direction methods that subsumes the classical (2-block) Alternating Direction Method of Multipliers. To specify the 2-block case, we refer to `classical ADMM'.} iterations obtained via duality. More recently, Alcantara, Dao, and Takeda~\cite{alcantaraADMM} have developed a new multiblock ADMM within a strongly convex-weakly convex setting which, in the standard convex case, reduces to applying duality to the splitting framework of~\cite{campoy}.

In recent years, different extensions of DR to deal with more than two operators have been proposed (see, e.g.,~\cite{campoy, condat2023proximal, malitsky2023resolvent, MBG24, ryu20, tam2023frugal}). To provide a unified perspective, many of them have recently been framed under the graph-based methodology introduced by Bredies, Chenchene, and Naldi~\cite{graph-drs}. This framework provides a systematic way to devise splitting algorithms by mapping the interactions between the operators onto the nodes and edges of a graph and a subgraph, where the underlying network topology determines how the variables are updated within the iteration. As is common in splitting methods, we distinguish between the governing sequence, which drives the iteration, and the associated shadow sequences obtained through the individual resolvent evaluations. When graph-based splitting is applied to a dual problem, we will refer to the latter as the dual shadow sequences.

Very recently, Alcantara, Dao and Takeda~\cite{ADT26composite} have introduced a general splitting framework for composite semimonotone inclusions with affine constraints, encompassing both multioperator and linearly constrained multiblock problems. Their approach handles the affine constraint directly within the splitting strategy, without employing duality. Therefore, the algorithms they derive for affine-constrained problems have a structure that differs from standard ADMM schemes. In particular, their two-block realization is explicitly shown to differ from classical ADMM, while their multiblock scheme is described as being reminiscent of a Jacobi-type ADMM method.

In this paper, we employ the complete Moreau's identity from \cite{gabayduality} along with the graph-based splitting framework of \cite{graph-drs} to construct a unified family of multiblock ADMM schemes. Our construction proceeds by applying graph-based splitting algorithms to a dual problem and transporting the resulting iterations back to the primal variables through Gabay duality. The resulting algorithms are parametrized by an algorithmic graph $G$, a connected spanning subgraph $G'$, and an onto decomposition of the Laplacian of the latter (see \cref{def:graph-setting}). Different choices of these graph structures lead to different coordination patterns among the variables, providing a common framework that encompasses several existing ADMM-type methods while also giving rise to new instances.

Our main contributions can be summarized as follows:
\begin{enumerate}[label={\rm (\roman*)}, leftmargin=*]
\item We systematically construct a novel, unified family of multiblock ADMM schemes from graph-based splitting methods applied to the dual problem. The roles of the graph $G$ and the onto decomposition of $G'$ are made explicit in the resulting iteration, showing how different graph topologies induce different coordination patterns among the variables. Several existing methods, including classical ADMM, the method of Malitsky--Tam, and that of Alcantara, Dao, and Takeda, are recovered as particular instances.
    
\item We study in detail the complete-graph instance of the proposed family. Complete connectivity is of particular practical interest, as densely connected graph structures have been observed to perform very well in numerical experiments (see, e.g.,~\cite{graph-fb}). We exploit this topology and obtain a compact explicit form of the resulting multiblock ADMM scheme which is ready for implementation.
    
\item We derive an equivalent reduced primal-dual formulation of the proposed family in which the governing sequence is eliminated, and the resulting iteration can be expressed in terms of the $n$ primal blocks and only $n-1$ independent dual variables. Remarkably, this reformulation no longer depends on the particular onto decomposition chosen. Besides providing an alternative representation of the resulting methods, this formulation further clarifies their connection with classical ADMM, where the retained dual shadow variable is identified with the usual ADMM Lagrange multiplier.
\end{enumerate}

The structure of the paper is as follows. \Cref{sec:preliminaries} introduces the notation and background used throughout the paper. In \cref{sec:graph-admm}, we construct the graph-based family of multiblock ADMM schemes, establish its convergence properties, and specialize the framework to convex minimization. \Cref{sec:instances} illustrates the framework on particular instances, first recovering several existing ADMM schemes, and then deriving a fully explicit method associated with the complete graph. In \cref{sec:multiplier-only}, we develop an equivalent primal--dual reformulation that clarifies the relationship with the classical setting of the ADMM. \cref{sec:numerics} reports numerical experiments comparing the practical performance of several instances of the family. Finally, \cref{sec:conclu} collects our concluding remarks.

\section{Preliminaries and notation}
\label{sec:preliminaries}

Throughout this paper, $\Gi$ and $\Hi_i$ (for each $i=1,\dots,n$) denote real Hilbert spaces, each endowed with its corresponding inner product $\langle \cdot, \cdot \rangle$ and induced norm $\|\cdot\|$. We write $x^k\to x$ and $x^k\rightharpoonup x$ to denote strong and weak convergence, respectively. We consider the product space $\bHi = \Hi_1 \times \dots \times \Hi_n$ equipped with the standard product metric. As a general notation rule, elements in product spaces will be marked in bold, e.g., $\bm{x}=(x_1,x_2,\ldots,x_n)\in\bHi$ where $x_i\in\Hi_i$ for each $i=1,\ldots,n$. We denote by $\bm{1}_n$ the all-ones vector in $\mathbb{R}^n$. With a slight abuse of notation, given a matrix $M\in\mathbb{R}^{p\times q}$, we use the same symbol $M$ for its canonical lifting \(M\otimes\Id_{\Gi}:\Gi^q\to\Gi^p\), where $\Id_{\Gi}$ denotes the identity mapping in $\Gi$. Thus, for every
$\bm{y}=(y_1,\ldots,y_q)\in\Gi^q$,
\[
    M\bm{y}
    =
    \left(\textstyle\sum_{j=1}^q M_{1,j}y_j,\sum_{j=1}^q M_{2,j}y_j,\ldots,\sum_{j=1}^q M_{p,j}y_j\right)\in\Gi^p.
\]
Similarly, vectors in $\R^n$ can be viewed as $n\times 1$ matrices, so they can be identified with their canonical lifting from $\Gi$ to $\Gi^n$. In particular, we regard $\bm 1_n$ also as the bounded linear embedding $\bm 1_n:\Gi\to\Gi^n$, i.e., \(\bm{1}_n y:=(y,\ldots,y)\in\Gi^n\) for any $y\in\Gi$, whose adjoint is precisely $\bm 1_n^*:\Gi^n\to\Gi$, which acts as $\bm 1_n^*\bm y=\sum_{i=1}^ny_i$, for $\bm y=(y_1,\ldots,y_n)\in\Gi^n$.

\subsection{Convex analysis and monotone operators}
In this section, generic Hilbert spaces $\Hi$ and $\Gi$ are used to state standard definitions and properties that naturally apply to any specific space in our framework. For a set $C$ in a Hilbert space, $\sri C$, $\ri C$ and $\operatorname{span}C$ stand for its \emph{strong relative interior}, its \emph{relative interior}, and the \emph{linear subspace generated} by $C$, respectively.

Let $F: \Hi \tto \Hi$ be a set-valued operator. The \emph{domain}, \emph{range}, and \emph{graph} of $F$ are defined respectively as
\begin{align*}
    \dom F &:= \{x \in \Hi \mid F(x) \neq \emptyset\}, \\
    \ran F &:= \{y \in \Hi \mid \exists x \in \Hi \text{ such that } y \in F(x)\}, \\
    \gra F &:= \{(x, y) \in \Hi \times \Hi \mid y \in F(x)\}.
\end{align*}
The \emph{inverse operator} $F^{-1}: \Hi \rightrightarrows \Hi$ is characterized by the equivalence $x \in F^{-1}(y) \Leftrightarrow y \in F(x)$.

An operator $F: \Hi \tto \Hi$ is said to be \emph{monotone} if
    \begin{equation*}
        \langle x - \tilde{x}, y - \tilde{y} \rangle \ge 0,\quad \forall (x, y), (\tilde{x}, \tilde{y}) \in \operatorname{gra} F.
    \end{equation*}
It is \emph{maximally monotone} if it is monotone and its graph cannot be strictly contained in the graph of any other monotone operator on $\Hi$. The \emph{resolvent} of $F$ is the operator $J_F: \Hi \tto \Hi$ defined as
\begin{equation*}
    J_F := (\Id_{\Hi} + F)^{-1}.
\end{equation*}
By Minty's Theorem, if $F$ is maximally monotone, then its resolvent $J_F$ is single-valued with full domain (see, e.g., \cite[Theorem 21.1]{bauschke}). Given $(\overline x,\overline y)\in\operatorname{gra}F$ and an increasing function $\phi:[0,+\infty[\,\to[0,+\infty]$, vanishing only at $0$, $F$ is \emph{uniformly monotone at $\overline x$} (relative to $\overline y$) \emph{with modulus} $\phi$ if
\begin{equation*}
    \langle x-\overline x,y-\overline y\rangle\geq\phi(\|x-\overline x\|),
    \quad\forall(x,y)\in\operatorname{gra}F.
\end{equation*}
When this holds with the same modulus $\phi$ at every $(\overline x,\overline y)\in\operatorname{gra}F$, then $F$ is said to be \emph{uniformly monotone with modulus} $\phi$. Given $\beta>0$, $F$ is \emph{$\beta$-strongly monotone} if $F-\beta\Id_\Hi$ is monotone, equivalently, if it is uniformly monotone with modulus $\phi(t)=\beta t^2$. Finally, given a nonempty subset $D$ of $\Hi$, a single-valued operator $T:D\to\Hi$ is \emph{$\beta$-cocoercive}, for $\beta>0$, if
\begin{equation*}
    \langle x-\tilde x,Tx-T\tilde x\rangle\geq\beta\|Tx-T\tilde x\|^2,
    \quad\forall x,\tilde x\in D.
\end{equation*}

We denote by $\mathcal{L}(\Hi, \Gi)$ the space of bounded linear operators from $\Hi$ into $\Gi$. For any operator $A \in \mathcal{L}(\Hi, \Gi)$, its \emph{adjoint operator} $A^* \in \mathcal{L}(\Gi, \Hi)$ is the unique operator satisfying
\begin{equation*}
    \langle Ax, y \rangle = \langle x, A^*y \rangle, \quad \forall x \in \Hi, \; \forall y \in \Gi.
\end{equation*}
The \emph{kernel} of $A$ is the closed linear subspace of $\Hi$ defined by $\ker A := \{x \in \Hi \mid Ax = 0\}$.  

Let $\Gamma_0(\Hi)$ denote the class of proper, convex, and lower semicontinuous functions from $\Hi$ to $]-\infty, +\infty]$. For any $f \in \Gamma_0(\Hi)$, its \emph{convex conjugate} $f^* \in \Gamma_0(\Hi)$ is defined as 
\[f^*(u) := \sup_{x \in \Hi} \{ \langle u, x \rangle - f(x) \},\quad \forall u \in \Hi.\] 
The \emph{subdifferential} of $f \in \Gamma_0(\Hi)$ is the set-valued mapping $\partial f: \Hi \rightrightarrows \Hi$ defined as
\[\partial f(x) := \{ y \in \Hi \mid f(\tilde{x}) \ge f(x) + \langle y, \tilde{x} - x \rangle, \; \forall \tilde{x} \in \Hi \},\quad \forall x\in\Hi.\]
A fundamental theorem due to Rockafellar states that if $f \in \Gamma_0(\Hi)$, then $\partial f$ is a maximally monotone operator (\cite[Theorem 21.2]{bauschke}), and its inverse mapping satisfies $(\partial f)^{-1} = \partial f^*$ (\cite[Corollary 16.30]{bauschke}). In addition, the resolvent of the subdifferential operator coincides with the classical \emph{proximal operator} $\operatorname{prox}_f: \Hi \to \Hi$, which can be characterized by
\[\operatorname{prox}_f(x) := J_{\partial f}(x) = \argmin_{y \in \Hi} \left\{ f(y) + \frac{1}{2}\|y - x\|^2 \right\},\quad \forall x\in\Hi;\]
see, e.g.,~\cite[Example~23.3]{bauschke}.

Given a nonempty subset $C$ of $\Hi$, the \emph{indicator function} of $C$ is the extended-real function $\iota_C:\Hi\to\,]-\infty,+\infty]$ defined as
\[ \iota_C(x)=
    \begin{cases}
        0, & \text{if } x\in C;\\
        +\infty, & \text{otherwise}.
    \end{cases}
\]
The \emph{normal cone} to $C$ at $x\in C$ is
\[N_C(x):=\{u\in\Hi\mid\langle z-x,u\rangle\leq0,\ \forall z\in C\},\]
and $N_C(x):=\emptyset$ when $x\not\in C$. If $C$ is nonempty, closed, and convex, $\iota_C\in\Gamma_0(\Hi)$, and $\partial\iota_C=N_C$ (see, e.g.,~\cite[Example~16.13]{bauschke}).

\subsection{Gabay duality and complete Moreau identities}
\label{sec:gabay-duality}

Gabay duality associates with a set-valued operator a dual operator whose resolvent can be expressed in terms of the original operator. We recall below the construction of the Gabay dual operator introduced in \cite[Definition~2.2]{gabayduality}, adapted to our setting.

\begin{definition}\label{def_D}
Let $F:\Hi\rightrightarrows\Hi$, $A\in\mathcal{L}(\Hi,\Gi)$ and $b\in\Gi$. The \emph{dual operator} associated with $(F,A,b)$ is the set-valued operator $D:\Gi\rightrightarrows\Gi$ defined by
\begin{equation}
\label{eq:dual-operator-general}
    D
    :=
    -A\circ F^{-1}\circ(-A^*)+b.
\end{equation}
\end{definition}

The following complete Moreau identity provides an expression
for the resolvent of $D$ in terms~of~$F$.

\begin{fact}[Complete Moreau's Identity {\cite[Theorem~3.1]{gabayduality}}]\label{thm:GMI}
Let $F:\Hi\rightrightarrows\Hi$, $A\in\mathcal{L}(\Hi,\Gi)$, $b\in\Gi$, and let $D$ be their associated dual operator as defined in \cref{eq:dual-operator-general}. Given $\gamma>0$,
\begin{equation}
\label{GMI_resolvent}
    J_{\gamma D}(y)
    =
    y-\gamma b
    +
    \gamma A
    \left(F+\gamma A^*A\right)^{-1}
    \left(-A^*\left(y- \gamma b\right)\right),\quad \forall y\in\Gi.
\end{equation}
Moreover, if $\lambda\in J_{\gamma D}(y)$,  then there exist
$v\in\Gi$ and $z\in\Hi$ such that
\begin{enumerate}[label={\rm(\roman*)},leftmargin=*]
    \item 
    $v\in D(\lambda)$ and $y=\lambda+\gamma v$; \label{item:GMI_1}

    \item
    $z\in F^{-1}(-A^*\lambda)$ and $v=-Az+b$; \label{item:GMI_2}

    \item
    $z \in (F+\gamma A^*A)^{-1} \bigl(-A^*(y-\gamma b)\bigr).$\label{item:GMI_3}
\end{enumerate}
\end{fact}

\begin{remark}[Classical Moreau identity]
\label{rem:Moreau}
If one specifies $F= \partial f$ for $f \in \Gamma_0(\Hi)$, $A=-\Id_\Hi$, and $b= 0$  in~\cref{thm:GMI}, then \cref{GMI_resolvent} becomes
\[\operatorname{prox}_{\gamma f^*}(y)
    =
    J_{\gamma\partial f^*}(y)
    =
    J_{\gamma(\partial f)^{-1}}(y)
    =
    y-\gamma
    (\partial f+\gamma\Id_{\Hi})^{-1}(y)
    =
    y-\gamma
    \operatorname{prox}_{f/\gamma}(y/\gamma),\]
which is the classical Moreau identity (see, e.g., \cite[Theorem 14.3]{bauschke}).
\end{remark}

We next give a sufficient condition for the maximal monotonicity of the dual operator.

\begin{fact}[{\cite[Proposition~20.22 and Theorem~25.3]{bauschke}}]
\label{lem:maximality-dual-operator}
Let $F:\Hi\rightrightarrows\Hi$ be maximally monotone, $A\in\mathcal{L}(\Hi,\Gi)$, $b\in\Gi$, and let $D$ be their associated dual operator as defined in \cref{eq:dual-operator-general}. If
\begin{equation}
\label{eq:CQ-dual-operator}
    0
    \in
    \sri 
    \left(
        \ran F
        -
        \ran A^*
    \right),
\end{equation}
then $D$ is maximally monotone.
\end{fact}

The following specialization illustrates the connection with convex duality.

\begin{remark}
\label{rem:dual-operator-convex}
Suppose that $F=\partial f$ for some $f\in\Gamma_0(\Hi)$, and assume that
\begin{equation}
\label{eq:CQ-dual-operator-convex}
    0 \in  \sri \left(\dom f^* - \ran A^* \right).
\end{equation}
Since $F^{-1}=(\partial f)^{-1}=\partial f^*$, the subdifferential chain rule (\cite[Theorem 16.47]{bauschke}) yields the equality $\partial\bigl(f^*\circ(-A^*)\bigr) = -A\circ\partial f^*\circ(-A^*)$. Thus, $D=\partial\left(f^*\circ(-A^*)+\langle b,\cdot\rangle\right)$ is maximally monotone. 

Condition \cref{eq:CQ-dual-operator-convex} is satisfied, for instance, whenever $\ran A^*\cap\int(\dom f^*)\neq\emptyset$, or, if $\Hi$ is finite-dimensional, when $\ran A^*\cap\ri (\dom f^*)\neq\emptyset$ (see, e.g., \cite[Proposition~6.19(vii)--(viii)]{bauschke}).
\end{remark}

\subsection{Graph theory}\label{sec: 2.3}

We say that $G=(\mathcal{N},\mathcal{E})$ is a (\emph{directed}) \emph{graph} if $\mathcal{N}$ is a finite set (whose elements are called \emph{nodes}) and $\mathcal{E}\subseteq\mathcal{N}\times\mathcal{N}$ (whose elements are called \emph{edges}). The \emph{order} of the graph is the number of its nodes, which we denote by $n=|\mathcal{N}|$. Further, we say that $G'=(\mathcal{N},\mathcal{E}')$ is a \emph{spanning subgraph} of $G$ if $\mathcal{E}'\subseteq\mathcal{E}$. The \emph{in-degree}, \emph{out-degree} and \emph{degree} are defined respectively as $d_i^{\rm in}:=|\{j\in\mathcal{N}:(j,i)\in\mathcal{E}\}|$, $d_i^{\rm out}:=|\{j\in\mathcal{N}:(i,j)\in\mathcal{E}\}|$ and $d_i:=d_i^{\rm in}+d_i^{\rm out}$. A \emph{path} in $G$ is a finite sequence of distinct nodes $(v_1,\ldots,v_r)$ with $r\geq2$ such that $(v_k,v_{k+1})\in\mathcal{E}$ or $(v_{k+1},v_{k})\in\mathcal{E}$  for all $k=1,\ldots,r-1$. The nodes $v_1$ and $v_r$ are called \emph{endpoints}. A graph $G$ is said to be \emph{connected} if for all distinct nodes $i,j\in\mathcal{N}$, there exists a path with $i$ and $j$ as endpoints. Further, $G$ is a \emph{tree} if it is connected and $|\mathcal{E}|=n-1$.

The \emph{degree matrix} of $G$ is the diagonal matrix $\operatorname{Deg}(G):=\Diag(d_1,\ldots,d_n)$. The \emph{degree balance} of $G$ is the vector $\bm{\delta}\in\mathbb{R}^n$ defined componentwise as
\[\delta_i:=d_i-2d_i^\text{in}=d_i^\text{out}-d_i^\text{in}, \quad \forall i\in\{1,\ldots,n\}.\]
The \emph{adjacency matrix} of $G$ is the matrix $\operatorname{Adj}(G)\in\mathbb{R}^{n\times n}$ whose components are defined~as
    \[\operatorname{Adj}(G)_{i,j}:=\begin{cases}
    	1 & \text{if }(i,j)\in\mathcal{E},\\
    	0 & \text{otherwise.}
    \end{cases}\]
The \emph{Laplacian matrix} of $G$ is the matrix $\operatorname{Lap}(G):=\operatorname{Deg}(G)-\operatorname{Adj}(G)-\operatorname{Adj}(G)^*$. Notice that, since the order of a graph is finite, so is the number of edges. Hence, we can also index $\mathcal{E}$ with natural numbers. Then, the \emph{incidence matrix} of $G$ is the matrix $\operatorname{Inc}(G)\in\mathbb{R}^{n\times |\mathcal{E}|}$ defined componentwise as
    \[\operatorname{Inc}(G)_{i,e}:=\begin{cases}
        1 & \text{if the edge }e\text{ leaves the node }i,\\
        -1 & \text{if the edge }e\text{ enters the node }i,\\
        0 & \text{otherwise}.
    \end{cases}\]

In this work, all graphs are assumed to satisfy the following structure, which is suitable for defining splitting algorithms. 

\begin{definition}\label{def: algorithmic graph}
We say that $G=(\mathcal{N},\mathcal{E})$ is an \emph{algorithmic graph} if $\mathcal{N}=\{1,\ldots,n\}$ with $n\geq2$, $(i,j)\in\mathcal{E}\Rightarrow i<j$, and $G$ is connected.
\end{definition}

\begin{fact}[{\cite[Fact 2.17]{graph-linear}}]\label{fact: Zdecom}
Let $G$ be an algorithmic graph of order $n$. The Laplacian matrix always has an onto decomposition, in the sense that 
\[\operatorname{Lap}(G)=ZZ^\ast,\quad\text{for some } Z\in \mathbb{R}^{n\times (n-1)}.\]
This implies $\rk Z=n-1$, $\ker Z=\{0\}$, and $\ker Z^\ast=\operatorname{span}\{\bm{1}_n\}$. 
\end{fact}

Numerically, an onto decomposition can always be computed by means of the spectral decomposition of the Laplacian. Under particular graph structures, sparser or more manageable decompositions can be obtained. For instance, we can take $Z=\Inc(G)$ whenever $G$ is a tree (see~\cite[Fact 2.17]{graph-linear} for further details).

\section{A graph-based family of multiblock ADMM algorithms}
\label{sec:graph-admm}

In this section, we construct a family of multiblock ADMM-type algorithms for solving~\cref{primal}. Our approach consists of associating our problem with a dual monotone inclusion problem where the graph splitting framework of~\cite{graph-drs} can be applied. 

\subsection{Construction of the dual inclusion}
\label{sec:dual-construction}

Recall that \cref{primal} is determined by maximally monotone operators \(F_i:\Hi_i\rightrightarrows\Hi_i\) and bounded linear operators \(A_i\in\mathcal{L}(\Hi_i,\Gi)\), for $i=1,\ldots,n$, together with a vector $b\in\Gi$.  Fix any $\bm{b}\in\Gi^n$ such that
\begin{equation}
\label{eq:decomposition-b}
  \bm{b}:=(b_1,\ldots,b_n),\quad\text{with }  \sum_{i=1}^n b_i=b,
\end{equation}
and, for each $i=1,\ldots,n$, consider the dual operator 
$D_i:\Gi\tto\Gi$ associated with $(F_i,A_i,b_i)$ as in \cref{def_D}; i.e.,
\begin{equation}
\label{eq:dual-operator}
    D_i
    :=
    -A_i\circ F_i^{-1}\circ(-A_i^*)+b_i.
\end{equation}
This leads to the following natural extension to $n$ operators of
the dual inclusion considered in
\cite[Theorem~2.3]{gabayduality}:
\begin{equation}
\tag{D}
\label{dual}
    \text{find }y\in\Gi
    \quad\text{such that}\quad
    0 \in\sum_{i=1}^nD_i(y).
\end{equation}

The following proposition relates the primal system
\cref{primal} to the dual inclusion \cref{dual}.

\begin{proposition}
\label{prop:primal-dual-equivalence}
Let $y\in\Gi$. Then $y$ solves \cref{dual} if and only if there
exists $\bm{x}=(x_1,\ldots,x_n)\in\bHi$ such that
$(\bm{x},y)$ solves \cref{primal}.
\end{proposition}

\begin{proof}
    This is a natural extension of the proof of \cite[Theorem~2.3]{gabayduality}.
A point $y\in\Gi$ solves \cref{dual} if and only if there exist
$p_i\in D_i(y)$, for $i=1,\ldots,n$, such that
\begin{equation}
\label{eq:dual-selections}
    \sum_{i=1}^n p_i=0.
\end{equation}
In view of \cref{eq:dual-operator}, the condition
$p_i\in D_i(y)$ is equivalent to the existence of some
$x_i\in\Hi_i$ satisfying
\[
    p_i=-A_ix_i+b_i,\quad \text{with } 0\in F_i(x_i)+A_i^*y,\quad \forall i=1,\ldots,n.
\]
Substituting these expressions for $p_i$ into
\cref{eq:dual-selections} and using
\cref{eq:decomposition-b} we obtain $\sum_{i=1}^nA_ix_i=b$, so it follows that $(\bm{x},y)$ solves \cref{primal}.
\end{proof}

\subsection{Derivation of the algorithm}
\label{sec:graph-admm-algorithm}

We impose the following standing assumption on the primal system \cref{primal}.

\begin{assumption}
\label{as:standing}
The following conditions hold:
\begin{enumerate}[leftmargin=1.8cm,label=(A\arabic*)]

    \item
    \label{A1}
    The primal system \cref{primal} has at least one solution.

    \item
    \label{A2}
    For every $i=1,\ldots,n$, the inclusion
    \( 0 \in \sri \left(\ran F_i - \ran A_i^* \right)\)
    holds.

\end{enumerate}
\end{assumption}

In addition to \cref{as:standing}, our algorithm relies on a choice of graph structures that determine the interactions between the variables. We fix this graph setting once, as follows.

\begin{definition}[Graph setting]
\label{def:graph-setting}
Let $G=(\mathcal N,\mathcal E)$ be an algorithmic graph of order $n$, and denote by $d_i$ the degree of node $i\in\mathcal N$. Let $G'=(\mathcal N,\mathcal E')$ be a connected spanning subgraph of $G$, and fix an onto decomposition
\begin{equation}
\label{eq:Z-decomposition}
    \operatorname{Lap}(G')=ZZ^*,
    \quad
    Z\in\R^{n\times(n-1)},
\end{equation}
as described in \cref{fact: Zdecom}. We refer to $(G,G',Z)$ as a
\emph{graph setting}.
\end{definition}

We are now in a position to state our graph-based multiblock ADMM in
\cref{alg:graphADMM}.

\begin{algorithm}[ht]
    \caption{Graph-based family of multiblock ADMM}
    \label{alg:graphADMM}
    \begin{algorithmic}[1]\setlength{\itemsep}{3pt}
        \State \textbf{input:}
problem data $(F_i,A_i)_{i=1}^n$ and $b$ of \cref{primal},
a decomposition $b=\sum_{i=1}^n b_i$, a graph setting
$(G,G',Z)$ as in \cref{def:graph-setting} (with degrees $d_1,\ldots,d_n$), a parameter $\tau>0$,
and a sequence of relaxation parameters
$(\theta_k)_{k\geq1}\subset{]\epsilon,2-\epsilon[}$ for some
$\epsilon\in]0,1[$

        \State \textbf{initialize:}
        $w_1^0,\ldots,w_{n-1}^0\in\Gi$

        \For{$k=0,1,2,\ldots$}
            \For{$i=1,\ldots,n$}

                \State
                \(
                    u_i^{k+1}
                    =
                    \sum_{j=1}^{n-1}Z_{i,j}w_j^k
                    +
                    2\sum_{(h,i)\in\mathcal E}y_h^{k+1}
                \)

                \State
                \(
                    x_i^{k+1}
                    \in
                    \left(
                        F_i+\frac{\tau}{d_i}A_i^*A_i
                    \right)^{-1}
                    \left(
                        -\frac1{d_i}
                        A_i^*(u_i^{k+1}-\tau b_i)
                    \right)
                \)
                \algcomment{primal sequence}

                \State
                \(
                    r_i^{k+1}
                    :=
                    A_ix_i^{k+1}-b_i
                \)
                \algcomment{residual sequence}

                \State
                \(
                    y_i^{k+1}
                    :=
                    \frac1{d_i}
                    \left(
                        u_i^{k+1}+\tau r_i^{k+1}
                    \right)
                \)
                \algcomment{dual shadow sequence}

            \EndFor

            \For{$j=1,\ldots,n-1$}

                \State
                \(
                    w_j^{k+1}
                    =
                    w_j^k
                    -
                    \theta_{k+1}
                    \sum_{i=1}^n Z_{i,j}y_i^{k+1}
                \)
                \algcomment{governing sequence}

            \EndFor
        \EndFor
    \end{algorithmic}
\end{algorithm}
We note that \cref{alg:graphADMM} is well defined. First, since $G$ is an algorithmic graph, the update of $u_i^{k+1}$ only involves
$y_h^{k+1}$ for $h<i$, which have already been computed in the
previous steps. It remains to check the existence of $x_i^{k+1}$ in step 6, which we show next.

\begin{lemma}[Well-posedness of \cref{alg:graphADMM}]
\label{lem:well-posed}
Suppose that \cref{as:standing} holds. Then, for each $i=1,\ldots,n$, there exists $x_i^{k+1}\in\Hi_i$ satisfying step 6 of \cref{alg:graphADMM}.
\end{lemma}

\begin{proof}
Fix $i\in\{1,\ldots,n\}$ and set $\gamma:=\tau / d_i$, and \(y:=u_i^{k+1} / d_i\).
By \cref{as:standing}\ref{A2} and
\cref{lem:maximality-dual-operator}, the dual operator $D_i$ is
maximally monotone. Hence, $J_{\gamma D_i}$ has full domain, so there
exists $\lambda\in J_{\gamma D_i}(y)$. Taking $y=u_i^{k+1}/d_i$ and $z=x^{k+1}_i$ in
\cref{thm:GMI}\ref{item:GMI_3} completes the proof.
\end{proof}

\begin{remark}[On the uniqueness of the generated sequences]\label{rem:uniq}
In general, $x_i^{k+1}$ satisfying step~6 of \cref{alg:graphADMM} need not be unique, which is why step~6 is stated as an inclusion rather than an equality. This is inherited from the possible non-uniqueness of the point $z$ in \cref{thm:GMI}\ref{item:GMI_3}. Uniqueness would hold, for instance, whenever $F_i+\frac{\tau}{d_i}A_i^*A_i$ is strictly monotone (in particular, if $A_i$ is injective). However, the feasibility residual 
\[r_i^{k+1}=A_ix_i^{k+1}-b_i,\]
appearing in step~7 of \cref{alg:graphADMM}, is always uniquely determined regardless of the
choice of~$x_i^{k+1}$. Indeed, if  $x,\widetilde x$ both satisfy step~6 of \cref{alg:graphADMM} for the same $u_i^{k+1}$, then, writing
$v:=-\frac{1}{d_i}A_i^*(u_i^{k+1}-\tau b_i)$, the monotonicity of $F_i$
gives
\[
    0
    \le
    \Big\langle x-\widetilde x,\
    \big(v-\tfrac{\tau}{d_i}A_i^*A_ix\big)
    -
    \big(v-\tfrac{\tau}{d_i}A_i^*A_i\widetilde x\big)
    \Big\rangle
    =
    -\tfrac{\tau}{d_i}\|A_i(x-\widetilde x)\|^2,
\]
forcing $A_ix=A_i\widetilde x$. Hence, although the primal iterates $x_i^{k+1}$ need not be unique, the residuals $r_i^{k+1}$ and, consequently, all subsequent auxiliary, dual and governing sequences are uniquely determined independently of this particular selection.
\end{remark}

\subsection{Convergence analysis}
\label{sec:convergence}

Before establishing the convergence of \cref{alg:graphADMM}, we show first its equivalence with the graph splitting method of \cite[\S~3]{graph-drs} applied to \cref{dual}. Throughout the paper, boldface notation in sequences is used for product vectors formed by block-indexed variables; for instance,
\[
    \bm y^k:=(y_i^k)_{i=1}^n,\quad
    \bm w^k:=(w_j^k)_{j=1}^{n-1},
\]
and analogously for the other block variables.

\begin{lemma}[Equivalence with the graph splitting method]
\label{lem:equiv}
The dual sequence $(\bm y^k)_{k\geq1}$ and the governing sequence $(\bm w^k)_{k\geq0}$ in \cref{alg:graphADMM} coincide with those generated by the graph splitting method of
\cite[Algorithm~3.1]{graph-drs} applied to $D_1,\ldots,D_n$ defined
in \cref{eq:dual-operator}, i.e.,
\begin{subequations}\label{eq:graph-splitting}
\begin{empheq}[left=\empheqlbrace]{align}
    y_i^{k+1}
    &=
    J_{\frac{\tau}{d_i}D_i}
    \left(
        \frac{2}{d_i}
        \textstyle\sum_{(h,i)\in\mathcal E}y_h^{k+1}
        +
        \frac{1}{d_i}
        \textstyle\sum_{j=1}^{n-1}Z_{i,j}w_j^k
    \right),
    \quad \text{for } i=1,\ldots,n,
    \label{eq:graph-splitting-y}
    \\
    \bm w^{k+1}
    &=
    \bm w^k-\theta_{k+1}Z^*\bm y^{k+1}.
    \label{eq:graph-splitting-w}
\end{empheq}
\end{subequations}
Consequently, under \cref{as:standing}, the following hold:
\begin{enumerate}[label={\rm(\roman*)}]
    \item
    $\bm w^k\wto\bm w^\star$ for some $\bm w^\star\in\Gi^{n-1}$;\label{lem:equiv_w}

    \item
    $\bm y^k\wto\bm1_ny^\star$ for some $y^\star\in\Gi$ such that $0\in\sum_{i=1}^nD_i(y^\star)$;\label{lem:equiv_y}

    \item
    $\|Z^*\bm y^k\|^2=o(1/k)$; \label{lem:equiv_Zy}

    \item
    $\|\bm y^k-\bm1_n\overline y^k\|^2=o(1/k)$, where $\overline y^k:=\tfrac1n\bm1_n^*\bm y^k$.    \label{lem:equiv_ybar}
\end{enumerate}
\end{lemma}

\begin{proof}
Fix $i\in\{1,\ldots,n\}$. Applying \cref{thm:GMI} with $D=D_i$, $\gamma=\tau/d_i$, and $y=u_i^{k+1}/d_i$, we obtain
\[
    J_{\frac{\tau}{d_i}D_i}
    \left(\frac{1}{d_i}u_i^{k+1}\right)
    =
    \frac{1}{d_i}u_i^{k+1}
    -\frac{\tau}{d_i}b_i
    +\frac{\tau}{d_i}A_ix_i^{k+1}
    =
    \frac{1}{d_i}
    \left(u_i^{k+1}+\tau r_i^{k+1}\right)
    =
    y_i^{k+1},
\]
where we have used the fact that 
\( x_i^{k+1} \in (F_i+\frac{\tau}{d_i}A_i^*A_i)^{-1} (-\frac{1}{d_i}A_i^*(u_i^{k+1}-\tau b_i))\),
together with the definitions of $r_i^{k+1}$ and $y_i^{k+1}$ in \cref{alg:graphADMM}. Combining this with the definition of $u_i^{k+1}$ yields~\cref{eq:graph-splitting-y}. On the other hand, the update of $w_j^{k+1}$ in \cref{alg:graphADMM} written in vector form directly gives \cref{eq:graph-splitting-w}. Therefore $(\bm y^k)_{k\geq1}$ and $(\bm w^k)_{k\geq0}$ coincide with the sequences generated by the graph splitting method in \cite[Algorithm~3.1]{graph-drs} applied to~\cref{dual}.

Suppose that \cref{as:standing} holds. Then by \ref{A1} and \cref{prop:primal-dual-equivalence}, the dual inclusion~\cref{dual} admits at least one solution. On the other hand, \ref{A2} and \cref{lem:maximality-dual-operator} imply that the operators $D_1,\ldots,D_n$ are maximally monotone. Hence, we can apply \cite[Theorem~3.2]{graph-drs} to obtain \ref{lem:equiv_w} and \ref{lem:equiv_y}. Moreover, \ref{lem:equiv_Zy} follows from \cite[Equation~(3.21)]{graph-drs}, whereas \cite[Proposition~3.5]{graph-drs} gives~\ref{lem:equiv_ybar}.
\end{proof}

We now turn to the convergence of the graph-based multiblock ADMM algorithm.

\begin{theorem}[Convergence of \cref{alg:graphADMM}]
\label{thm:convergence}
Suppose that \cref{as:standing} holds, and let the sequences
$(\bm x^k)_{k\geq1}\subseteq\bHi$,
$(\bm w^k)_{k\geq0}\subseteq\Gi^{n-1}$,
$(\bm r^k)_{k\geq1}\subseteq\Gi^n$, and
$(\bm y^k)_{k\geq1}\subseteq\Gi^n$
be generated by \cref{alg:graphADMM}. Then there exist
$y^\star\in\Gi$ solving \cref{dual} and
$\bm w^\star\in\Gi^{n-1}$ such that:
\begin{enumerate}[label={\rm(\roman*)}]
\item
\(\bm w^k\rightharpoonup\bm w^\star\), and \(\|\bm w^{k+1}-\bm w^k\|^2=o(1/k)\); \label{thm:convergence_w}

\item
\(\bm y^k\wto\bm1_ny^\star\); \label{thm:convergence_y}

\item
$\|\bm1_n^*\bm r^k\|^2 = \|\sum_{i=1}^n A_ix_i^k-b\|^2=o(1/k)$; \label{thm:convergence_r}

\item
for every weak sequential cluster point $\bm x$ of $(\bm x^k)_{k\geq1}$ the pair $(\bm x,y^\star)$ solves
\cref{primal}; \label{thm:convergence_x}

\item
there exists $\bm x^\star\in\bHi$ such that $(\bm x^\star,y^\star)$ solves \cref{primal} and $A_ix_i^k\wto A_ix_i^\star$, for all $i=1,\ldots,n$. \label{thm:convergence_z}
\end{enumerate}
\end{theorem}

\begin{proof}
\ref{thm:convergence_w}:
By \cref{lem:equiv}\ref{lem:equiv_w}, $\bm w^k\rightharpoonup\bm w^\star$ for some $\bm w^\star\in\Gi^{n-1}$. Moreover, from the update of $\bm w^k$,
\[\bm w^{k+1}-\bm w^k = -\theta_{k+1}Z^*\bm y^{k+1}.\]
Hence, using \cref{lem:equiv}\ref{lem:equiv_Zy} and the boundedness of $(\theta_k)_{k\geq1}$, we get $\|\bm w^{k+1}-\bm w^k\|^2=o(k^{-1})$.
\medskip

\ref{thm:convergence_y}:
The result follows directly from \cref{lem:equiv}\ref{lem:equiv_y}.
\medskip

\ref{thm:convergence_r}:
Let $\bm\delta$ be the degree balance vector of $G$. Since $\bm1_n^*\bm\delta=0$, we have that $\bm\delta^*    \left(\bm y^k-\bm1_n\overline y^k\right)=\bm\delta^*\bm y^k$. Thus, \cref{lem:equiv}\ref{lem:equiv_ybar} yields
\begin{equation}
\label{eq:strong-convergence}
    \|\bm\delta^*\bm y^k\|^2=o(1/k).
\end{equation}
On the other hand, from the updates of $u_i^{k+1}$ and $y_i^{k+1}$ in \cref{alg:graphADMM},
\[
    d_i y_i^{k+1}
    =
    \sum_{j=1}^{n-1}Z_{i,j}w_j^k
    +
    2\sum_{(h,i)\in\mathcal E}y_h^{k+1}
    +
    \tau r_i^{k+1},
    \quad \forall i=1,\ldots,n.
\]
Summing these identities over $i$ and using $\bm1_n^*Z=0$ (since $\ker Z^*=\operatorname{span}\{\bm1_n\}$ by \cref{fact: Zdecom}), we obtain
\[
    \tau\bm1_n^*\bm r^{k+1}
    =
    \sum_{i=1}^n d_i y_i^{k+1}
    -
    2\sum_{i=1}^n
    \sum_{(h,i)\in\mathcal E}y_h^{k+1}.
\]
Observe that
\(
    \sum_{i=1}^n
    \sum_{(h,i)\in\mathcal E}y_h^{k+1}
    =
    \sum_{i=1}^n d_i^{\rm out}y_i^{k+1}
\), so it follows that
\begin{equation}
\label{eq:residual-identity}
    \bm1_n^*\bm r^{k+1}
    = \frac1\tau \sum_{i=1}^n(d_i-2d_i^{\rm out})y_i^{k+1}
    = \frac1\tau \sum_{i=1}^n(-\delta_i)y_i^{k+1}
    = -\frac1\tau\bm\delta^*\bm y^{k+1}.
\end{equation}
By \cref{eq:strong-convergence} and recalling $\sum_{i=1}^nb_i=b$, \ref{thm:convergence_r} follows.
\medskip

\ref{thm:convergence_x}:
Rewriting the $\bm x^{k+1}$ update step in \cref{alg:graphADMM}, and using $r_i^{k+1}$ and $y_i^{k+1}$, we obtain
\begin{equation}\label{eq:monotone-inclusion_system}
    0
    \in
    F_i(x_i^{k+1})
    +
    \frac1{d_i}A_i^*
    \left(u_i^{k+1}+\tau r_i^{k+1}\right)
    =
    F_i(x_i^{k+1})+A_i^*y_i^{k+1},
    \quad\forall i=1,\ldots,n.
\end{equation}
Define the product operators $\mathcal F:\bHi\rightrightarrows\bHi$ and $\mathcal A:\bHi\to\Gi^n$ given by
\begin{equation}
\label{eq:F-A-block-notation}
    \mathcal F(\bm x):=F_1(x_1)\times\cdots\times F_n(x_n)
    \quad\text{and}\quad
    \mathcal A\bm x:=(A_1x_1,\ldots,A_nx_n),
\end{equation}
for any $\bm x=(x_1,\ldots,x_n)\in\bHi$. With this notation, the system of inclusions \cref{eq:monotone-inclusion_system} reads as 
\begin{equation}
\label{eq:monotone-inclusion}
    \bm0
    \in
    \mathcal F(\bm x^{k+1})+\mathcal A^*\bm y^{k+1}.
\end{equation}
Since $\bm r^{k+1}=\mathcal A\bm x^{k+1}-\bm b$ and $\bm1_n^*\bm b=b$, \cref{eq:residual-identity} can be written as $\bm1_n^*\mathcal A\bm x^{k+1}+\frac1\tau\bm\delta^*\bm y^{k+1}=b$. Using this, together with \cref{eq:monotone-inclusion}, we obtain
\begin{equation*}
    \begin{bmatrix}
        \bm0 \\ b
    \end{bmatrix}
    \in
    \begin{bmatrix}
        \mathcal F & \mathcal A^* \\
        \bm1_n^*\mathcal A & \frac1\tau\bm\delta^*
    \end{bmatrix}
    \begin{bmatrix}
        \bm x^{k+1} \\ \bm y^{k+1}
    \end{bmatrix}.
\end{equation*}
Recalling $\overline y^k=\frac1n\bm1_n^*\bm y^k$, we can rewrite the previous inclusion as
\begin{equation}\label{eq:block-inclusion}
    \begin{bmatrix}
        \mathcal A^*
        \left(\bm1_n\overline y^{k+1}-\bm y^{k+1}\right) \\
        -b+\frac1\tau\bm\delta^*\bm y^{k+1}
    \end{bmatrix}
    \in
    \begin{bmatrix}
        \mathcal F & \mathcal A^*\bm1_n \\
        -\bm1_n^*\mathcal A & 0
    \end{bmatrix}
    \begin{bmatrix}
        \bm x^{k+1} \\ \overline y^{k+1}
    \end{bmatrix}.
\end{equation}
The previous block operator is maximally monotone as the sum of a maximally monotone operator and a skew-symmetric linear operator (see, e.g., \cite[Proposition~20.23, Example~20.35 and Corollary~25.5(i)]{bauschke}). Hence, its graph is sequentially closed in $(\bHi\times\Gi)^{\text{weak}}\times
(\bHi\times\Gi)^{\text{strong}}$ (see, e.g., \cite[Proposition~20.38(ii)]{bauschke}). By \cref{lem:equiv}\ref{lem:equiv_ybar} and \cref{eq:strong-convergence}, the left-hand side of \cref{eq:block-inclusion} converges strongly to $(\bm 0,-b)$. Moreover, by \ref{thm:convergence_y}, $\overline y^k\wto y^\star$. Let $\bm x$ be any weak sequential cluster point of $(\bm x^k)_{k\geq1}$. Then there exists a subsequence such that $(\bm x^{k_j},\overline y^{k_j}) \wto(\bm x,y^\star)$, so by demiclosedness of the operator,
\[
    \begin{bmatrix}
        \bm0 \\ -b
    \end{bmatrix}
    \in
    \begin{bmatrix}
        \mathcal F & \mathcal A^*\bm1_n \\
        -\bm1_n^*\mathcal A & 0
    \end{bmatrix}
    \begin{bmatrix}
        \bm x \\ y^\star
    \end{bmatrix}
    =
    \begin{bmatrix}
        \mathcal F(\bm x)+\mathcal A^*(\bm1_ny^\star) \\
        -\bm1_n^*\mathcal A\bm x
    \end{bmatrix}.
\]
The first block proves $0\in F_i(x_i)+A_i^*y^\star$ for $i=1,\ldots,n$, whereas the second gives $\sum_{i=1}^nA_ix_i=b$. Altogether, we get that $(\bm x,y^\star)$ solves \cref{primal}.
\medskip

\ref{thm:convergence_z}:
From the updates of $u_i^k$ and $y_i^k$ in \cref{alg:graphADMM},
\[
    r_i^k=\frac1\tau\left(d_iy_i^k-2\sum_{(h,i)\in\mathcal E}y_h^k
    -\sum_{j=1}^{n-1}Z_{i,j}w_j^{k-1}\right),\quad\forall i=1,\ldots,n.
\]
Since $\bm y^k\wto\bm1_ny^\star$ and $\bm w^k\wto\bm w^\star$ by \ref{thm:convergence_y} and \ref{thm:convergence_w}, we get $\bm r^k\wto\bm r^\star$, with $\bm{r}^\star:=\frac1\tau(\bm{\delta}y^\star-Z\bm{w}^\star)$.

Set $\bm p^k:=-\bm r^k$ and $\bm p^\star:=-\bm r^\star$. By \cref{eq:dual-operator} and \cref{eq:monotone-inclusion_system}, $(\bm y^k,\bm p^k)\in\gra\mathcal D$, where $\mathcal D:=D_1\times\cdots\times D_n$ is maximally monotone on $\Gi^n$. By \ref{thm:convergence_r}, $\sum_{i=1}^np_i^k=-\bm1_n^*\bm r^k\to0$, whence $\sum_{i=1}^np_i^\star=0$ and, in particular, $\langle\bm1_ny^\star,\bm p^\star\rangle =\langle y^\star,\sum_{i=1}^np_i^\star\rangle=0$. Moreover,
\begin{equation}\label{mu->0}
    \langle\bm y^k,\bm p^k\rangle
    =\langle \bm y^k - \bm{1}_n\overline{y}^k , \bm p^k \rangle + \langle \bm{1}_n\overline{y}^k , \bm p^k \rangle = \langle \bm y^k - \bm{1}_n\overline{y}^k , \bm p^k \rangle + \langle \overline{y}^k , \bm{1}_n^*\bm p^k \rangle \to 0,
\end{equation}
since $\bm y^k - \bm{1}_n\overline{y}^k\to0$ by \cref{lem:equiv}\ref{lem:equiv_ybar}, $(\bm p^k)_{k\geq1}$ is bounded (being weakly convergent), and $(\overline y^k)_{k\geq1}$ is bounded. Therefore, \cref{mu->0} implies that $\langle\bm y^k,\bm p^k\rangle\to\langle\bm1_ny^\star,\bm p^\star\rangle$ and \cite[Corollary~20.59(ii)]{bauschke} yields $\bm p^\star\in\mathcal D(\bm1_ny^\star)$, i.e., $p_i^\star\in D_i(y^\star)$ for every $i=1,\ldots,n$.

Hence, by \cref{eq:dual-operator}, for each $i$ there exists $x_i^\star\in\Hi_i$ with $-A_i^*y^\star\in F_i(x_i^\star)$ and $p_i^\star=-A_ix_i^\star+b_i$. Consequently $A_ix_i^k=b_i-p_i^k\wto b_i-p_i^\star=A_ix_i^\star$, and $\sum_{i=1}^np_i^\star=0$ gives $\sum_{i=1}^nA_ix_i^\star=b$. Altogether, $(\bm x^\star,y^\star)$ solves \cref{primal}.
\end{proof}

\begin{remark}[Strong convergence of the dual sequence]
\label{rem:strong-convergence-y}
Suppose in addition that, for some $i\in\{1,\ldots,n\}$, the dual operator $D_i$ is uniformly monotone on every bounded subset of $\dom D_i$. Then, by \cite[Proposition~3.10]{graph-drs},
\[\bm y^k\to\bm1_n y^\star.\]
That is, the convergence in \cref{lem:equiv}\ref{lem:equiv_y} and, consequently, that of \cref{thm:convergence}\ref{thm:convergence_y}, is strong. Moreover, if the corresponding modulus of uniform monotonicity can be chosen of the form $\phi(t)=\mu t^2$ for some $\mu>0$ (in particular, if $D_i$ is strongly monotone), then by \cite[Remark~3.11]{graph-drs}
\[\|y_i^k-y^\star\|^2=o(k^{-1/2}).\]

A sufficient condition for the strong monotonicity of $D_i$ can be given directly in terms of the primal data; namely, when $F_i$ is $\beta$-cocoercive for some $\beta>0$ and $A_i$ is surjective. Indeed, the latter is equivalent to the existence of $c>0$ such that
\begin{equation}
\label{eq:Ajstar-bounded-below}
    \|A_i^*y\|\geq c\|y\|,
    \quad \forall y\in\Gi;
\end{equation}
see, e.g., \cite[Fact~2.26]{bauschke}. Let $(y,v),(\widehat y,\widehat v)\in\gra D_i$. By \cref{eq:dual-operator}, there exist $z,\widehat z\in\Hi_i$ such that \(-A_i^*y\in F_i(z)\) and \(-A_i^*\widehat y\in F_i(\widehat z)\), with \(v=-A_iz+b_i\) and \(\widehat v=-A_i\widehat z+b_i\). Therefore, by the $\beta$-cocoercivity of $F_i$ and \cref{eq:Ajstar-bounded-below},
\begin{align*}
    \langle y-\widehat y,v-\widehat v\rangle
    &=
    \langle
        z-\widehat z,
        -A_i^*(y-\widehat y)
    \rangle
    \geq
    \beta\|A_i^*(y-\widehat y)\|^2
    \geq
    \beta c^2\|y-\widehat y\|^2.
\end{align*}
Thus, $D_i$ is $\beta c^2$-strongly monotone, and the strong convergence above applies.
\end{remark}

Note that \cref{thm:convergence}\ref{thm:convergence_z} only guarantees that $A_ix_i^k\wto A_ix_i^\star$ for every $i=1,\ldots,n$, which does not yield convergence of $x_i^k$ without further assumptions. This limitation is standard in the ADMM literature, even for the classical two-block ADMM (see, e.g., \cite[Corollary~5.2]{bartz2022adaptive}). We present next some sufficient conditions under which the primal sequence does itself converge.

\begin{theorem}[Convergence of the primal sequence] \label{thm:x-convergence}
Consider the setting of \cref{thm:convergence}, and let $\bm x^\star=(x_1^\star,\ldots,x_n^\star)$ be as in \cref{thm:convergence}\ref{thm:convergence_z}. For each fixed $i\in\{1,\ldots,n\}$ the following hold:
\begin{enumerate}[label={\rm(\roman*)}]
    \item\label{thm:x-conv-a}
    if $A_i$ is injective with closed range, then \(x_i^k\wto x_i^\star\);

    \item\label{thm:x-conv-b}
    if $F_i$ is uniformly monotone at $x_i^\star$ (relative to $-A_i^*y^\star$), then \(x_i^k\to x_i^\star\).
\end{enumerate}
\end{theorem}

\begin{proof}
\ref{thm:x-conv-a}: Since $A_i$ has closed range, we consider the Moore--Penrose pseudoinverse $A_i^\dagger:\Gi\to\Hi_i$, which is linear and bounded. Since $A_i$ is also injective, $A_i^\dagger A_i=\Id_{\Hi_i}$ (see \cite[Remark 2.7]{graph-linear}). Thus, applying the pseudoinverse in \cref{thm:convergence}\ref{thm:convergence_z} gives
\[x_i^k=A_i^\dagger A_ix_i^k \wto A_i^\dagger A_ix_i^\star=x_i^\star.\]

\ref{thm:x-conv-b}:
Since $(\bm x^\star,y^\star)$ solves \cref{primal}, then $-A_j^*y^\star\in F_j(x_j^\star)$ for every $j=1,\ldots,n$ and $\sum_{j=1}^nA_jx_j^\star=b$. We know from  \cref{eq:monotone-inclusion_system} that $-A_j^*y_j^k\in F_j(x_j^k)$, for all $j=1,\ldots,n$ and all $k\geq1$. By applying the monotonicity of $F_j$ to these two pairs, we obtain
\begin{equation*}
    0\leq\eta_j^k
    :=
    \left\langle x_j^k-x_j^\star,\,-A_j^*(y_j^k-y^\star)\right\rangle
    =
    \left\langle r_j^\star-r_j^k,\,y_j^k-y^\star\right\rangle,
    \quad \forall j=1,\ldots,n,\ \forall k\geq1.
\end{equation*}
Also, recalling that $\sum_{j=1}^nr_j^\star=0$, we note that $\sum_{j=1}^n\eta_j^k = \langle\bm r^\star,\bm y^k-\bm 1_n y^\star\rangle -\langle\bm r^k,\bm y^k\rangle + \langle \bm 1_n^*\bm r^k,y^\star\rangle$. By \cref{mu->0}, and \cref{thm:convergence}\ref{thm:convergence_y} and \ref{thm:convergence_r}, we get $\sum_{j=1}^n\eta_j^k\to0$. As every $\eta_j^k$ is nonnegative, we conclude that $\bm \eta^k\to0$.
 
Finally, let $\phi_i$ be a modulus of uniform monotonicity of $F_i$ at $x_i^\star$ relative to $-A_i^*y^\star$. Applying it to $(x_i^k,-A_i^*y_i^k)\in\gra F_i$ gives $\phi_i\bigl(\|x_i^k-x_i^\star\|\bigr)\leq\eta_i^k$. Given $\varepsilon>0$, $\phi_i(\varepsilon)>0$ and, by the convergence of $\bm \eta^k$, the inequality $\eta_i^k<\phi_i(\varepsilon)$ eventually holds. Since $\phi_i$ is increasing, $\|x_i^k-x_i^\star\|<\varepsilon$. As this is true for all $\varepsilon>0$,  $x_i^k\to x_i^\star$.
\end{proof}

\subsection{The case of convex minimization}
\label{sec:convex-minimization}

Throughout this section, we specialize to the case where, for each $i=1,\ldots,n$, $F_i$ is the subdifferential of a proper, convex, and lower semicontinuous function $f_i:\Hi_i\to\,]-\infty,+\infty]$. Although the results of this section admit natural (but technical) extensions to general Hilbert spaces, for simplicity of the resulting conditions, we restrict to the finite-dimensional setting $\Gi=\R^m$ and $\Hi_i=\R^{\ell_i}$, for $i=1,\ldots,n$. Thus, we focus on the problem
\begin{equation}\tag{P\textsubscript{opt}}\label{primal_opt}
 	\min\; \sum_{i=1}^nf_i(x_i)\quad\text{s.t. }\sum_{i=1}^nA_ix_i=b,
\end{equation}
where $b\in\R^m$, $f_i\in\Gamma_0(\R^{\ell_i})$ and $A_i\in\R^{m\times \ell_i}$, for each $i=1,\ldots,n$.

We begin by relating, under some constraint qualification, the system \cref{primal}  to the first-order optimality conditions of \cref{primal_opt}.

\begin{proposition}[Primal optimality] \label{prop:primal-optimality}
Let $f_i\in\Gamma_0(\R^{\ell_i})$, for $i=1,\ldots,n$, and consider $F_i=\partial f_i$ in \cref{primal}. The following hold:
\begin{enumerate}[label={\rm(\roman*)}]
    \item If $(\bm x,y)$ solves \cref{primal}, then $\bm x$ solves \cref{primal_opt}.\label{primal-opt-a} 
    \item Suppose that there exist $\widehat x_1\in\R^{\ell_1},\ldots,\widehat x_n\in\R^{\ell_n}$ such that \label{primal-opt-b}
    \begin{equation}
    \sum_{i=1}^nA_i\widehat x_i=b\quad\text{and}\quad\widehat x_i\in
    \ri (\dom f_i),\quad \forall i=1,\ldots,n.
    \end{equation}
    If $\bm x$ solves \cref{primal_opt}, then there exists $y\in\R^m$ such that $(\bm x,y)$ solves \cref{primal}.
\end{enumerate}
\end{proposition}

\begin{proof}
Define $h:\R^{\ell_1}\times\cdots\times\R^{\ell_n}\to{]-\infty,+\infty]}$ and $S:\R^{\ell_1}\times\cdots\times\R^{\ell_n}\to\R^m$  as
\begin{equation}\label{eq:h-S}
h(\bm z):=\sum_{i=1}^nf_i(z_i),\quad\text{and}\quad S(\bm z):=\sum_{i=1}^nA_iz_i,
\end{equation}
and set $C:=S^{-1}(b)$, so that \cref{primal_opt} reads as minimizing $h+\iota_C$. By Fermat's rule,
\begin{equation}\label{eq:Popt-as-Fermat}
    \bm x\text{ solves }\cref{primal_opt}\quad\Leftrightarrow\quad 0\in\partial(h+\iota_C)(\bm x).
\end{equation}
Since $C$ is an affine subspace with $C-C=\ker S$, we have $\partial\iota_C(\bm x)=N_C(\bm x)=(\ker S)^\perp=\ran S^*$ for every $\bm x\in C$ (see, e.g.,~\cite[Example~6.43 and Fact~2.25(iv)]{bauschke}). We deduce that
\begin{equation}\label{eq:P-as-Fermat}
    \exists y\in\R^m\text{ such that }(\bm x,y)\text{ solves \cref{primal}}
    \quad\Leftrightarrow\quad
    0\in\partial h(\bm x)+\partial\iota_C(\bm x).
\end{equation}

\labelcref{primal-opt-a}:
By \cite[Theorem~23.8]{rockafellar1970convex}, $\partial h+\partial\iota_C\subseteq\partial(h+\iota_C)$, so \cref{eq:P-as-Fermat} implies \cref{eq:Popt-as-Fermat}.
\medskip

\labelcref{primal-opt-b}:
Note that $\ri (\dom h)=\prod_{i=1}^n\ri (\dom f_i)$ and $\ri C=C$. Thus, we get $\widehat{\bm x}\in\ri (\dom h)\cap\ri C$. This condition provides (see, e.g., \cite[Theorem~23.8]{rockafellar1970convex}) the sum rule \(\partial(h+\iota_C)=\partial h+\partial{\iota_C}\) and \cref{eq:Popt-as-Fermat} implies \cref{eq:P-as-Fermat}.
\end{proof}

Having related \cref{primal} to \cref{primal_opt}, we now turn to the dual side. The Fenchel--Rockafellar dual of \cref{primal_opt} takes the form
\begin{equation}
\tag{D\textsubscript{opt}}
\label{dual_opt}
    \min \; \sum_{i=1}^n f_i^*(-A_i^*y) + \langle b,y\rangle.
\end{equation}
The following proposition shows that, under a suitable constraint qualification, \cref{dual_opt} is, analogously, an instance of the dual inclusion \cref{dual}.

\begin{proposition}[Dual optimality]
\label{prop:convex-dual}
Let $f_i\in\Gamma_0(\R^{\ell_i})$ and recall the dual operators $D_i$ defined in \cref{eq:dual-operator}, with $F_i=\partial f_i$, for $i=1,\ldots,n$. The following hold:
\begin{enumerate}[label={\rm(\roman*)}]
    \item If $y\in\R^m$ solves \cref{dual}, then $y$ solves \cref{dual_opt}.\label{convex-dual-a}
    \item Suppose that there exists $\widehat y\in\R^m$ such that\label{convex-dual-b}
    \begin{equation}
    \label{eq:CQ-dual-finite}
        -A_i^*\widehat y \in  \ri (\dom f_i^*), \quad \forall i=1,\ldots,n.
    \end{equation}
    If $y$ solves \cref{dual_opt}, then $y$ solves \cref{dual}.
\end{enumerate}
\end{proposition}
\begin{proof}
For each $i=1,\ldots,n$, write $g_i:=f_i^*\circ(-A_i^*)$ and set $g:=\sum_{i=1}^ng_i+\langle b,\cdot\rangle$, so that \cref{dual_opt} amounts to minimizing $g$. By Fermat's rule, $y$ solves \cref{dual_opt} if and only if $0\in\partial g(y)$. Since $(\partial f_i)^{-1}=\partial f_i^*$ and $\sum_{i=1}^nb_i=b$, recalling \cref{eq:dual-operator} we get that $y$ solves \cref{dual} if and only if $0\in \sum_{i=1}^n \bigl((-A_i)\circ\partial f_i^*\bigr)(-A_i^*y)+b$.
\medskip

\labelcref{convex-dual-a}:
By \cite[Theorems~23.8 and~23.9]{rockafellar1970convex}, $(-A_i)\circ\partial f_i^*\circ(-A_i^*)\subseteq\partial g_i$ for all $i=1,\ldots,n$ and $\sum_{i=1}^n\partial g_i+b\subseteq\partial g$. Thus, $\sum_{i=1}^n \bigl((-A_i)\circ\partial f_i^*\circ(-A_i^*)\bigr)+b\subseteq\partial g$ and \cref{dual} implies \cref{dual_opt}.
\medskip

\labelcref{convex-dual-b}:
Condition \cref{eq:CQ-dual-finite} implies that $\ran (-A_i^*)\cap\ri (\dom f_i^*)\neq\emptyset$ for all $i=1,\ldots,n$. Hence, the subdifferential chain rule \cite[Theorem~23.9]{rockafellar1970convex} gives
\begin{equation}
\label{eq:dual-individual-chain-rule}
    \partial g_i  =  -A_i\circ\partial f_i^*\circ(-A_i^*).
\end{equation}
Moreover, \(\dom g_i = (-A_i^*)^{-1}(\dom f_i^*)\), where the inverse denotes the preimage of a set. Since \cref{eq:CQ-dual-finite} implies that \((-A_i^*)^{-1} \bigl(\ri (\dom f_i^*)\bigr)\neq\varnothing\), \cite[Theorem~6.7]{rockafellar1970convex} yields
\[\ri (\dom g_i) = (-A_i^*)^{-1} \bigl(\ri (\dom f_i^*)\bigr),\quad \forall i=1,\ldots,n.\]
In particular, we obtain that \(\widehat y \in \bigcap_{i=1}^n \ri (\dom g_i)\neq\emptyset\). The subdifferential sum rule \cite[Theorem~23.8]{rockafellar1970convex}, together with \cref{eq:dual-individual-chain-rule}, therefore gives $\sum_{i=1}^n \bigl((-A_i)\circ\partial f_i^*\circ(-A_i^*)\bigr)+b=\partial g$ and \cref{dual_opt} implies \cref{dual}.
\end{proof}

It is often useful to characterize the solutions of \cref{primal_opt} and \cref{dual_opt} as saddle points of a suitable Lagrangian function.

\begin{definition}[Augmented Lagrangian] \label{def:lagrangian}
Given $\gamma\geq0$, the \emph{augmented Lagrangian} associated with \cref{primal_opt} is the function $L_\gamma:\bigl(\R^{\ell_1}\times\cdots\times\R^{\ell_n}\bigr)\times\R^m\to\,]-\infty,+\infty]$,
\begin{equation*}
    L_\gamma(\bm x,y):=\sum_{i=1}^nf_i(x_i)
    +
    \left\langle y,\sum_{i=1}^nA_ix_i-b\right\rangle
    +
    \frac\gamma2\left\|\sum_{i=1}^nA_ix_i-b\right\|^2.
\end{equation*}
\end{definition}

\begin{definition}[Saddle point] \label{def:saddle-point}
Given $\gamma\geq0$, a pair $(\bm x,y)\in\bigl(\R^{\ell_1}\times\cdots\times\R^{\ell_n}\bigr)\times\R^m$ is a \emph{saddle point} of $L_\gamma$ if
\begin{equation}\label{eq:saddle-point}
    L_\gamma(\bm x,\eta)\leq L_\gamma(\bm x,y)\leq L_\gamma(\bm\xi,y),
    \quad
    \forall\bm\xi\in\R^{\ell_1}\times\cdots\times\R^{\ell_n},
    \ \forall\eta\in\R^m.
\end{equation}
\end{definition}

\begin{proposition}[Characterization of saddle points] \label{prop:saddle-points}
Fix $\gamma\geq0$ and let $L_\gamma$ be the augmented Lagrangian associated with \cref{primal_opt}. Then $(\bm x,y)$ is a saddle point of $L_\gamma$ if and only if $(\bm x,y)$ solves \cref{primal} (with $F_i=\partial f_i$, for $i=1,\ldots,n$). In that case, $\bm x$ solves \cref{primal_opt} and $y$ solves \cref{dual_opt}.
\end{proposition}

\begin{proof}
Let $h$ and $S$ be the mappings defined in \cref{eq:h-S} so that the augmented Lagrangian can be evaluated as
\[
    L_\gamma(\bm x,y)=h(\bm x)
        +
        \left\langle y,S\bm x-b\right\rangle
        +
        \frac\gamma2\left\|S\bm x-b\right\|^2.
\]
By \cite[Theorem 1 and (2.19)]{Rockafellar1970Monotone}, $(\bm x,y)$ is a saddle point of $L_\gamma$ if and only if it satisfies the first order conditions $0\in \partial L_\gamma(\cdot,y)(\bm x)$ and $0\in \partial (-L_\gamma(\bm x,\cdot))(y)$. Since $h$ is proper, convex and lower semicontinuous, then
\[ \partial L_\gamma(\cdot,y)(\bm x)=\partial h(\bm x)+S^*y+\gamma S^*(S\bm x-b),\quad \partial (-L_\gamma(\bm x,\cdot))(y)=-(S\bm x-b). \]
Thus, the first order conditions are equivalent to $S\bm x = b$ and $0\in \partial h(\bm x)+S^*y$, i.e.,\cref{primal}.

Finally, \cref{prop:primal-optimality}\ref{primal-opt-a} shows that $\bm x$ solves \cref{primal_opt}. Moreover, by \cref{prop:primal-dual-equivalence}, $y$ solves \cref{dual}, and hence \cref{prop:convex-dual}\ref{convex-dual-a} shows that $y$ solves \cref{dual_opt}.
\end{proof}

We can specialize \cref{alg:graphADMM} to this context of convex minimization. To this end, we make use of the following lemma.

\begin{lemma}
\label{lem:x-update-argmin}
Let $f\in\Gamma_0(\R^\ell)$, $A\in\mathcal L(\R^\ell,\R^m)$, and $\gamma>0$.
Then,
\[
    \left(\partial f+\gamma A^*A\right)^{-1}(-\gamma A^*v)
    =
    \operatorname*{argmin}_{x\in\R^\ell}
    \left\{
        f(x)+\frac\gamma2\|Ax+v\|^2
    \right\}, \quad \forall v\in\R^m.
\]
\end{lemma}

\begin{proof}
For all $x\in\R^\ell$, the inclusion \(x\in\left(\partial f+\gamma A^*A\right)^{-1}(-\gamma A^*v)\) is equivalent to the inclusion \(0\in\partial f(x)+\gamma A^*(Ax+v)\). Since $x\mapsto\frac\gamma2\|Ax+v\|^2$ is everywhere differentiable with gradient $\gamma A^*(Ax+v)$, and has full domain, the sum rule (see, e.g.,~\cite[Corollary 16.48(iii)]{bauschke}) gives us \(0\in\partial\left(f+\frac\gamma2\|A\cdot+v\|^2\right)(x)\), which holds if and only if $x$ minimizes $f+\frac\gamma2\|A\cdot+v\|^2$ over $\R^\ell$.
\end{proof}

Applying the previous result, we obtain the specialization of \cref{alg:graphADMM} for solving \cref{primal_opt}, described in \cref{alg:graphADMM-convex}.

\begin{algorithm}[ht]
    \caption{Graph-based family of multiblock ADMM (convex minimization case)}
    \label{alg:graphADMM-convex}
     \begin{algorithmic}[1]\setlength{\itemsep}{3pt}
        \State \textbf{input:}
problem data $(f_i,A_i)_{i=1}^n$ and $b$ of \cref{primal_opt},
a decomposition $b=\sum_{i=1}^n b_i$, a graph setting
$(G,G',Z)$ as in \cref{def:graph-setting} (with degrees $d_1,\ldots,d_n$), a parameter $\tau>0$,
and a sequence of relaxation parameters
$(\theta_k)_{k\geq1}\subset{]\epsilon,2-\epsilon[}$ for some
$\epsilon\in]0,1[$

        \State \textbf{initialize:}
        $w_1^0,\ldots,w_{n-1}^0\in\mathbb{R}^m$

        \For{$k=0,1,2,\ldots$}
            \For{$i=1,\ldots,n$}

                \State
                \(
                    u_i^{k+1}
                    =
                    \sum_{j=1}^{n-1}Z_{i,j}w_j^k
                    +
                    2\sum_{(h,i)\in\mathcal E}y_h^{k+1}
                \)

                \State
                \(
                    x_i^{k+1}
                    \in
                    \operatorname*{argmin}_{x\in\mathbb{R}^{\ell_i}}
                    \left\{
                        f_i(x)
                        +
                        \dfrac{\tau}{2d_i}
                        \left\|
                            A_ix-b_i+\tfrac1\tau u_i^{k+1}
                        \right\|^2
                    \right\}
                \)
                \algcomment{primal sequence}

                \State
                \(
                    r_i^{k+1}
                    :=
                    A_ix_i^{k+1}-b_i
                \)
                \algcomment{residual sequence}

                \State
                \(
                    y_i^{k+1}
                    :=
                    \frac1{d_i}
                    \left(
                        u_i^{k+1}+\tau r_i^{k+1}
                    \right)
                \)
                \algcomment{dual shadow sequence}

            \EndFor

            \For{$j=1,\ldots,n-1$}

                \State
                \(
                    w_j^{k+1}
                    =
                    w_j^k
                    -
                    \theta_{k+1}
                    \sum_{i=1}^n Z_{i,j}y_i^{k+1}
                \)
                \algcomment{governing sequence}

            \EndFor
        \EndFor
    \end{algorithmic}
\end{algorithm}

\begin{remark}[Augmented-Lagrangian interpretation]
The primal updates of \cref{alg:graphADMM-convex} admit a local augmented-Lagrangian interpretation. Indeed, for each $i=1,\ldots,n$, define
\[
    \rho_i:=\frac{\tau}{d_i},
    \quad
    \lambda_i^k
    :=
    \frac1{d_i}
    \sum_{j=1}^{n-1}Z_{i,j}w_j^k,
    \quad\text{and}\quad
    q_i^{k+1}
    :=
    \frac2\tau
    \sum_{(h,i)\in\mathcal E}y_h^{k+1}.
\]
Then the update of $x_i^{k+1}$ can equivalently be written as
\[
    x_i^{k+1}
    \in
    \argmin_{x\in\mathbb{R}^{\ell_i}}
    \left\{
        f_i(x)
        +
        \left\langle
            \lambda_i^k,A_ix-b_i
        \right\rangle
        +
        \frac{\rho_i}{2}
        \left\|
            A_ix-b_i+q_i^{k+1}
        \right\|^2
    \right\}.
\]
Thus, each primal step has the form of an augmented-Lagrangian minimization with a block-dependent penalty and a graph-dependent shift. In general, however, $q_i^{k+1}$ need not coincide with the sum of the residuals of the remaining primal blocks.
\end{remark}

\begin{assumption}[Convex setting]
\label{as:convex}
Let $f_i\in\Gamma_0(\R^{\ell_i})$ and $A_i\in\R^{m\times\ell_i}$, for $i=1,\ldots,n$, and let $b\in\R^m$ defining \cref{primal_opt}. The following conditions hold:
\begin{enumerate}[leftmargin=1.8cm,label={\rm(B\arabic*)}]
    \item
    \label{B1}
    At least one of the following is satisfied:
    \begin{enumerate}[label={\rm(\alph*)}]
        \item
        \label{B1a}
        the augmented Lagrangian $L_\gamma$ has a saddle point for some $\gamma\geq0$;

        \item
        \label{B1b}
        \cref{primal_opt} has a solution and there exists
        $\widehat{\bm x}=(\widehat x_1,\ldots,\widehat x_n)$ with
        $\sum_{i=1}^nA_i\widehat x_i=b$ and $\widehat x_i\in\ri (\dom f_i)$, for all $i=1,\ldots,n$;

        \item
        \label{B1c}
        \cref{dual_opt} has a solution and there exists $\widehat y\in\R^m$ with
        $-A_i^*\widehat y\in\ri (\dom f_i^*)$, for all $i=1,\ldots,n$.
    \end{enumerate}

    \item
    \label{B2}
    For each $i\in\{1,\ldots,n\}$,  at least one of the following is satisfied:
    \begin{enumerate}[label={\rm(\alph*)}]
    \item
    \label{B2a}
    $A_i$ is injective;

    \item
    \label{B2b}
    $f_i$ is strongly convex;

    \item
    \label{B2c}
    \(\ran A_i^* \cap \ri (\dom f_i^*) \neq\emptyset\).
    \end{enumerate}
\end{enumerate}
\end{assumption}

The next lemma shows the identification between \cref{as:convex} and \cref{as:standing}.

\begin{lemma}
\label{lem:convex-CQ}
Let $F_i=\partial f_i$ for $i=1,\ldots,n$.
\begin{enumerate}
    \item\label{B1->A1} Any item in \cref{B1} implies \ref{A1}.
    \item\label{B2<->A2} Each of \ref{B2}\ref{B2a} and \ref{B2}\ref{B2b} implies \ref{B2}\ref{B2c}, which is equivalent to \ref{A2}.
    \item\label{B1c->B2} \ref{B1}\ref{B1c} implies \ref{B2}.
\end{enumerate}
\end{lemma}

\begin{proof}
\ref{B1->A1}: 
Suppose that \cref{B1}\ref{B1a} holds, i.e., $L_\gamma$ has a saddle point $(\bm x,y)$. By \cref{prop:saddle-points}, $(\bm x,y)$ solves \cref{primal}. If instead \cref{B1}\ref{B1b} holds, then, in view of \cref{prop:primal-optimality}\ref{primal-opt-b}, there exists $y$ such that $(\bm x,y)$ solves \cref{primal}, where $\bm x$ is the assumed solution of \cref{primal_opt}. Finally, if \cref{B1}\ref{B1c} holds, then $y$ solves \cref{dual} by \cref{prop:convex-dual}\ref{convex-dual-b}. Therefore, by \cref{prop:primal-dual-equivalence}, there exists $\bm x$ such that $(\bm x,y)$ solves \cref{primal}.
\medskip

\ref{B2<->A2}:
Fix $i\in\{1,\ldots,n\}$. If $A_i$ is injective, then $\ran A_i^*=\R^{\ell_i}$. If instead $f_i$ is strongly convex, then $\dom f_i^*=\R^{\ell_i}$ (see, e.g., \cite[Proposition 14.15]{bauschke}). In both cases, \ref{B2}\ref{B2c} trivially holds.

On the other hand, since $\Hi_i=\mathbb{R}^{\ell_i}$ is finite-dimensional and $\ran A_i^*$ is a subspace, \ref{A2} amounts to $\ran A_i^*\cap\ri (\ran\partial f_i)\neq\emptyset$ (\cite[Fact~6.14(i) and Corollary~6.15(ii)]{bauschke}). Since $\ri (\dom f_i^*)\subseteq\ran\partial f_i\subseteq\dom f_i^*$ (\cite[Theorem~23.4 and Corollary~23.5.1]{rockafellar1970convex}), we get $\ri (\ran\partial f_i)=\ri (\dom f_i^*)$ (\cite[Corollary~6.3.1]{rockafellar1970convex}), which proves the equivalence \ref{A2}$\Leftrightarrow$\ref{B2}\ref{B2c}.
\medskip

\ref{B1c->B2}:
It suffices to see that $-A_i^*\widehat y\in\ran A_i^*\cap\ri (\dom f_i^*)$.
\end{proof}

\begin{corollary}[Convergence of \cref{alg:graphADMM-convex}]
\label{cor:convergence-convex}
Suppose that \cref{as:convex} holds, and let $(\bm x^k)_{k\geq1}$ and $(\bm y^k)_{k\geq1}$ be generated by
\cref{alg:graphADMM-convex}. The following hold:
\begin{enumerate}[label={\rm(\roman*)}]
    \item \label{cor:conv_I}
    There exists a solution $y^\star$ of \cref{dual_opt} such that \(\bm y^k\to\bm1_ny^\star\).    

    \item \label{cor:conv_II}
    The feasibility gap vanishes; i.e., $\sum_{i=1}^nA_ix_i^k\to b$.
    
    \item \label{cor:conv_III}
    Every sequential cluster point of $(\bm x^k)_{k\geq1}$ solves \cref{primal_opt}.
    
    \item \label{cor:conv_IV}
    There exists $\bm x^\star=(x_1^\star,\ldots,x_n^\star)$ solving
    \cref{primal_opt} such that $A_ix_i^k\to A_ix_i^\star$ for all $i=1,\ldots,n$.
\end{enumerate}
Furthermore, if for each $i=1,\ldots,n$ either \cref{B2}\cref{B2a} or \cref{B2}\cref{B2b} holds, then $\bm x^k\to\bm x^\star$.
\end{corollary}

\begin{proof}
Since \cref{as:convex} holds, by \cref{lem:x-update-argmin} the sequences generated by \cref{alg:graphADMM-convex} coincide with those generated by \cref{alg:graphADMM}. Since all the spaces involved are finite-dimensional, weak and strong convergence coincide, so \cref{thm:convergence}\ref{thm:convergence_y} and \cref{prop:convex-dual}\ref{convex-dual-a} give \ref{cor:conv_I}; \cref{thm:convergence}\ref{thm:convergence_r} gives \ref{cor:conv_II}; and \cref{thm:convergence}\ref{thm:convergence_x} and \ref{thm:convergence_z}, combined with \cref{prop:primal-optimality}\ref{primal-opt-a}, give \ref{cor:conv_III} and \ref{cor:conv_IV}, respectively.

Finally, if $A_i$ is injective, then the convergence follows from \cref{thm:x-convergence}\ref{thm:x-conv-a}, since the range in finite dimensions is always closed. On the other hand, suppose that $f_i$ is strongly convex. Then $\partial f_i$ is strongly monotone at $x_i^\star$, and \cref{thm:x-convergence}\ref{thm:x-conv-b} yields the convergence.
\end{proof}

\section{Some specific instances of the algorithm}
\label{sec:instances}

We now turn to
particular instances of our family of algorithms obtained by specializing the graph setting of \cref{def:graph-setting}.

\subsection{Recovering some known ADMM algorithms}
\label{subsec:knownalgorithms}

We first show how several multiblock ADMM schemes already present in the literature (namely, classical ADMM, the method of Malitsky--Tam, and that of Alcantara, Dao, and Takeda) are recovered as particular cases of \cref{alg:graphADMM-convex}.

\begin{example}[Classical ADMM]\label{ex:classADMM}
We begin with the classical two-block ADMM~\cref{eq:clasADMM}. Consider the problem
\begin{equation}
\label{eq:two-block-problem}
    \min_{x_1\in\mathbb{R}^{\ell_1},\,x_2\in\mathbb{R}^{\ell_2}}
    f_1(x_1)+f_2(x_2)
    \quad\text{s.t.}\quad
    A_1x_1+A_2x_2=b.
\end{equation}
Since $n=2$, we consider the unique algorithmic graph $G$ of order two, with $\mathcal E=\{(1,2)\}$, and take $G'=G$. In particular, $d_1=d_2=1$. We choose the onto decomposition of $\operatorname{Lap}(G')$ given by
\begin{equation}
\label{eq:Z-classADMM}
    Z=
    \begin{bmatrix}
        1\\-1
    \end{bmatrix}.
\end{equation}

To recover the usual order of the classical ADMM iteration, we apply \cref{alg:graphADMM-convex} to the reordered blocks $(f_2,A_2,b)$ and $(f_1,A_1,0)$, and take $\theta_k=1$ for all $k\geq1$. For convenience of notation, we retain the original block indices for the variables; that is, the variables associated with the first node are indexed by $2$, while those of the second node, by $1$. With this convention, $u_2^{k+1}=w^k$ and $u_1^{k+1}=-w^k+2y_2^{k+1}$, so \cref{alg:graphADMM-convex} becomes
\begin{equation}
\label{eq:classical-admm-graph-form}
\left\{
\begin{aligned}
    x_2^{k+1}
    &\in
    \operatorname*{argmin}_{x_2\in\mathbb{R}^{\ell_2}}
    \left\{
        f_2(x_2)
        +
        \frac{\tau}{2}
        \left\|
            A_2x_2-b+\frac1\tau w^k
        \right\|^2
    \right\},
    \\
    y_2^{k+1}
    &:=
    w^k+\tau (A_2x_2^{k+1}-b),
    \\
    x_1^{k+1}
    &\in
    \operatorname*{argmin}_{x_1\in\mathbb{R}^{\ell_1}}
    \left\{
        f_1(x_1)
        +
        \frac{\tau}{2}
        \left\|
            A_1x_1+\frac1\tau (-w^k+2y_2^{k+1})
        \right\|^2
    \right\},
    \\
    y_1^{k+1}
    &:=
    -w^k+2y_2^{k+1}+\tau A_1x_1^{k+1},
    \\
    w^{k+1}
    &=
    w^k-y_2^{k+1}+y_1^{k+1}.
\end{aligned}
\right.
\end{equation}
We next eliminate the governing sequence $(w^k)_k$. From the last two updates in \cref{eq:classical-admm-graph-form}, we obtain that
\begin{equation}
\label{eq:classical-admm-w-shadow}
    w^k
    =
    y_2^k+\tau A_1x_1^k, \quad \forall k\geq1.
\end{equation}
Moreover, the $y_2^{k+1}$ update yields $-w^k+2y_2^{k+1}=y_2^{k+1}+\tau(A_2x_2^{k+1}-b)$. Substituting this identity, together with \cref{eq:classical-admm-w-shadow} in \cref{eq:classical-admm-graph-form}, we obtain
\begin{equation*}
\left\{
\begin{aligned}
    x_2^{k+1}
    &\in
    \operatorname*{argmin}_{x_2\in\mathbb{R}^{\ell_2}}
    \left\{
        f_2(x_2)
        +
        \frac{\tau}{2}
        \left\|
            A_1x_1^k+A_2x_2-b+\frac1\tau y_2^k
        \right\|^2
    \right\},
    \\
    y_2^{k+1}
    &:=
    y_2^k+\tau (A_1x_1^k+A_2x_2^{k+1}-b),
    \\
    x_1^{k+1}
    &\in
    \operatorname*{argmin}_{x_1\in\mathbb{R}^{\ell_1}}
    \left\{
        f_1(x_1)
        +
        \frac{\tau}{2}
        \left\|
            A_1x_1+A_2x_2^{k+1}-b+\frac{1}{\tau}{y_2^{k+1}}
        \right\|^2
    \right\},
    \\
    y_1^{k+1}
    &:=
    y_2^{k+1}+\tau(A_1x_1^{k+1}+A_2x_2^{k+1}-b).
\end{aligned}
\right.
\end{equation*}
Observe that $y_1^{k+1}$ is no longer needed for the other updates, so we remove it from the iteration. Expanding the squared norms and discarding independent terms, we obtain the iteration
\begin{equation}
\label{eq:classical-admm-cyclic}
\left\{
\begin{aligned}
    x_2^{k+1}
    &\in
    \operatorname*{argmin}_{x_2\in\mathbb{R}^{\ell_2}}
    \left\{
        f_2(x_2)
        +
        \langle y_2^k,A_2x_2\rangle
        +
        \frac{\tau}{2}
        \left\|
            A_1x_1^k+A_2x_2-b
        \right\|^2
    \right\},
    \\
    y_2^{k+1}
    &=
    y_2^k
    +
    \tau
    \left(
        A_1x_1^k+A_2x_2^{k+1}-b
    \right),
    \\
    x_1^{k+1}
    &\in
    \operatorname*{argmin}_{x_1\in\mathbb{R}^{\ell_1}}
    \left\{
        f_1(x_1)
        +
        \langle y_2^{k+1},A_1x_1\rangle
        +
        \frac{\tau}{2}
        \left\|
            A_1x_1+A_2x_2^{k+1}-b
        \right\|^2
    \right\}.
\end{aligned}
\right.
\end{equation}
Finally, by starting this cyclic iteration at the $x_1$-update and relabeling $x_2^{k+1}$ as $x_2^k$ and $y_2^{k+1}$ as $y^k$, we recover the standard form of ADMM in \cref{eq:clasADMM}.
\end{example}

\begin{example}[Malitsky--Tam multiblock ADMM]
We next show that the multiblock ADMM proposed by Malitsky and Tam~\cite[Algorithm~3]{malitsky2023resolvent} is also recovered as a particular instance of \cref{alg:graphADMM-convex}. For $n\geq3$, let $G'=(\mathcal N,\mathcal E')$ be the sequential graph, with $\mathcal N=\{1,\ldots,n\}$ and
\[\mathcal E' = \{(i,i+1):i=1,\ldots,n-1\},\]
and let $G=(\mathcal N,\mathcal E)$ be the ring graph, where $\mathcal E = \mathcal E'\cup\{(1,n)\}$. Observe that $d_i=2$, for all $i=1,\ldots,n$. Since $G'$ is a tree, we choose its incidence matrix as the onto decomposition of $\Lap(G')$, given componentwise by
\begin{equation}
\label{eq:Z-MT}
    Z_{i,j}
    =
    \begin{cases}
        1,  & \text{if } i=j,\\
        -1, & \text{if } i=j+1,\\
        0,  & \text{otherwise}.
    \end{cases}
\end{equation}
The decomposition of $b$ is taken as $(b_1,\ldots,b_n)=(0,\ldots,0,b)$.

From the structure of $G$ and \cref{eq:Z-MT}, the updates of the auxiliary variables in \cref{alg:graphADMM-convex} take the form
\begin{equation}
\label{eq:u-MT-recursive}
\left\{
\begin{aligned}
    u_1^{k+1}
    &=
    w_1^k,\\
    u_i^{k+1}
    &=
    w_i^k-w_{i-1}^k+2y_{i-1}^{k+1},
    \quad i=2,\ldots,n-1,\\
    u_n^{k+1}
    &=
    -w_{n-1}^k+2y_{n-1}^{k+1}+2y_1^{k+1}.
\end{aligned}
\right.
\end{equation}
For $i=1,\ldots,n-1$, the update of the corresponding dual variables becomes \(2y_i^{k+1} = u_i^{k+1}+\tau A_ix_i^{k+1}.\) Thus, by induction in \cref{eq:u-MT-recursive}, we obtain
\begin{equation}
\label{eq:u-MT}
\left\{
\begin{aligned}
    u_1^{k+1}
    &=
    w_1^k,\\
    u_i^{k+1}
    &=
    w_i^k
    +
    \tau\textstyle\sum_{j=1}^{i-1}A_jx_j^{k+1},
    \quad i=2,\ldots,n-1,\\
    u_n^{k+1}
    &=
    w_1^k
    +
    \tau
    \left(
        2A_1x_1^{k+1}
        +
        \textstyle\sum_{j=2}^{n-1}A_jx_j^{k+1}
    \right).
\end{aligned}
\right.
\end{equation}
Consequently, the primal updates in \cref{alg:graphADMM-convex} reduce to
\begin{equation}
\label{eq:MT-primal-updates}
\left\{
\begin{aligned}
    x_1^{k+1}
    &\in
    \operatorname*{argmin}_{x\in\mathbb{R}^{\ell_1}}
    \left\{
        f_1(x)
        +
        \frac{\tau}{4}
        \left\|
            A_1x+\tfrac1\tau w_1^k
        \right\|^2
    \right\},
    \\
    x_i^{k+1}
    &\in
    \operatorname*{argmin}_{x\in\mathbb{R}^{\ell_i}}
    \left\{
        f_i(x)
        +
        \frac{\tau}{4}
        \left\|
            \textstyle\sum_{j=1}^{i-1}A_jx_j^{k+1}
            +A_ix+\frac1\tau w_i^k
        \right\|^2
    \right\},
    \quad i=2,\ldots,n-1,
    \\
    x_n^{k+1}
    &\in
    \operatorname*{argmin}_{x\in\mathbb{R}^{\ell_n}}
    \left\{
        f_n(x)
        +
        \frac{\tau}{4}
        \left\|
            2A_1x_1^{k+1}
            +
            \textstyle\sum_{j=2}^{n-1}A_jx_j^{k+1}
            +A_nx-b
            +\frac1\tau w_1^k
        \right\|^2
    \right\}.
\end{aligned}
\right.
\end{equation}

On the other hand, \cref{eq:u-MT} and the dual updates give
\begin{equation}
\label{eq:y-MT}
    2y_i^{k+1}
    = \left\{\begin{array}{ll}
    w_i^k
    +
    \tau\sum_{j=1}^iA_jx_j^{k+1},
    & \text{if } i\in\{1,\ldots,n-1\},\\
    w_1^k
    +
    \tau
    \left(
        2A_1x_1^{k+1}
        +
        \sum_{j=2}^nA_jx_j^{k+1}
        -b
    \right), & \text{if } i=n.
    \end{array}\right.
\end{equation}
Using that \((Z^*\bm y^{k+1})_i = y_i^{k+1}-y_{i+1}^{k+1}\), for all $i=1,\ldots,n-1$, together with
\cref{eq:y-MT}, the update of the governing sequence reduces to
\begin{equation}
\label{eq:w-MT}
\left\{
\begin{aligned}
    w_i^{k+1}
    &=
    w_i^k
    +
    \frac{\theta_{k+1}}{2}
    \left(
        w_{i+1}^k-w_i^k
    \right)
    +
    \frac{\theta_{k+1}\tau}{2}
    A_{i+1}x_{i+1}^{k+1},
    \quad i=1,\ldots,n-2,
    \\
    w_{n-1}^{k+1}
    &=
    w_{n-1}^k
    +
    \frac{\theta_{k+1}}{2}
    \left(
        w_1^k-w_{n-1}^k
    \right)
    +
    \frac{\theta_{k+1}\tau}{2}
    \left(
        A_1x_1^{k+1}
        +
        A_nx_n^{k+1}
        -b
    \right).
\end{aligned}
\right.
\end{equation}
Finally, define $v_i^k:=\frac1\tau w_i^k$, for each $i=1,\ldots,n-1$. Taking $\tau=2$ and $\theta_k=2\gamma$ for all $k\geq1$, with $\gamma\in(0,1)$, \cref{eq:MT-primal-updates} and \cref{eq:w-MT} become
\begin{equation*}
\label{eq:MT-final}
\left\{
\begin{aligned}
    x_1^{k+1}
    &\in
    \operatorname*{argmin}_{x\in\mathbb{R}^{\ell_1}}
    \left\{
        f_1(x)
        +
        \frac12
        \left\|
            A_1x+v_1^k
        \right\|^2
    \right\},
    \\
    x_i^{k+1}
    &\in
    \operatorname*{argmin}_{x\in\mathbb{R}^{\ell_i}}
    \left\{
        f_i(x)
        +
        \frac12
        \left\|
            \textstyle\sum_{j=1}^{i-1}A_jx_j^{k+1}
            +A_ix+v_i^k
        \right\|^2
    \right\},
    \quad i=2,\ldots,n-1,
    \\
    x_n^{k+1}
    &\in
    \operatorname*{argmin}_{x\in\mathbb{R}^{\ell_n}}
    \left\{
        f_n(x)
        +
        \frac12
        \left\|
            2A_1x_1^{k+1}
            +
            \textstyle\sum_{j=2}^{n-1}A_jx_j^{k+1}
            +A_nx-b
            +v_1^k
        \right\|^2
    \right\},
    \\
    v_i^{k+1}
    &=
    v_i^k
    +
    \gamma
    \left(
        v_{i+1}^k-v_i^k
    \right)
    +
    \gamma A_{i+1}x_{i+1}^{k+1},
    \quad i=1,\ldots,n-2,
    \\
    v_{n-1}^{k+1}
    &=
    v_{n-1}^k
    +
    \gamma
    \left(
        v_1^k-v_{n-1}^k
    \right)
    +
    \gamma
    \left(
        A_1x_1^{k+1}
        +
        A_nx_n^{k+1}
        -b
    \right),
\end{aligned}
\right.
\end{equation*}
which is precisely the Malitsky--Tam multiblock ADMM in \cite[Algorithm~3]{malitsky2023resolvent}.
\end{example}

\begin{example}[Alcantara--Dao--Takeda multiblock ADMM]
We finally show that a particular convex instance of the multiblock ADMM proposed by Alcantara, Dao, and Takeda~\cite[Algorithm~2]{alcantaraADMM} is also recovered from \cref{alg:graphADMM-convex}.

For $n\geq3$, let $G=G'=(\mathcal N,\mathcal E)$ be the parallel-up
graph, where
\[
    \mathcal N=\{1,\ldots,n\},
    \quad
    \mathcal E=\{(1,i):i=2,\ldots,n\}.
\]
Thus, $G$ is a tree with $d_1=n-1$ and $d_i=1$ for $i=2,\ldots,n$. We choose its incidence matrix as the onto decomposition of $\Lap(G)$, namely,
\begin{equation}
\label{eq:Z-ADT}
    Z=
    \begin{bmatrix}
        \bm1_{n-1}^*\\
        -I_{n-1}
    \end{bmatrix}.
\end{equation}
We apply \cref{alg:graphADMM-convex} to the reordered blocks \((f_n,A_n,b),(f_1,A_1,0),\ldots,(f_{n-1},A_{n-1},0)\). As in the classical ADMM case, for convenience of notation we retain the original block indices for the primal and dual variables. Thus, the variables associated with the first node are indexed by $n$, while those associated with nodes $2,\ldots,n$ are indexed by $1,\ldots,n-1$, respectively.

From \cref{eq:Z-ADT} and the structure of the parallel-up graph, the auxiliary variables satisfy
\[
    u_n^{k+1}
    =\sum_{j=1}^{n-1}w_j^k
    ,
    \quad\text{and}\quad
    u_i^{k+1}
    =
    -w_i^k+2y_n^{k+1},
    \quad \forall i=1,\ldots,n-1.
\]
Consequently, \cref{alg:graphADMM-convex} becomes
\begin{equation}
\label{eq:ADT-pre-shift}
\left\{
\begin{aligned}
    x_n^{k+1}
    &\in
    \operatorname*{argmin}_{x\in\mathbb{R}^{\ell_n}}
    \left\{
        f_n(x)
        +
        \frac{\tau}{2(n-1)}
        \left\|
            A_nx-b+\textstyle\frac1\tau\sum_{j=1}^{n-1}w_j^k
        \right\|^2
    \right\},
    \\
    y_n^{k+1}
    &=
    \frac1{n-1}
    \left(
        \textstyle\sum_{j=1}^{n-1}w_j^k
        +
        \tau(A_nx_n^{k+1}-b)
    \right),
    \\
    x_i^{k+1}
    &\in
    \operatorname*{argmin}_{x\in\mathbb{R}^{\ell_i}}
    \left\{
        f_i(x)
        +
        \frac{\tau}{2}
        \left\|
            A_ix+\tfrac1\tau(-w_i^k+2y_n^{k+1})
        \right\|^2
    \right\},
    \quad i=1,\ldots,n-1,
    \\
    y_i^{k+1}
    &=
    -w_i^k+2y_n^{k+1}
    +
    \tau A_ix_i^{k+1},
    \quad i=1,\ldots,n-1,
    \\
    w_i^{k+1}
    &=
    w_i^k-\theta_{k+1}
    \left(
        y_n^{k+1}-y_i^{k+1}
    \right),
    \quad i=1,\ldots,n-1.
\end{aligned}
\right.
\end{equation}
Cyclically shifting the iteration so that the first $n-1$ primal blocks are computed before the $n$th block, the previous scheme is equivalently written as
\begin{equation}
\label{eq:ADT-post-shift}
\left\{
\begin{aligned}
    x_i^{k+1}
    &\in
    \operatorname*{argmin}_{x\in\mathbb{R}^{\ell_i}}
    \left\{
        f_i(x)
        +
        \frac{\tau}{2}
        \left\|
            A_ix+\tfrac1\tau(-w_i^k+2y_n^k)
        \right\|^2
    \right\},
    \quad i=1,\ldots,n-1,
    \\
    y_i^{k+1}
    &=
    -w_i^k+2y_n^k
    +
    \tau A_ix_i^{k+1},
    \quad i=1,\ldots,n-1,
    \\
    w_i^{k+1}
    &=
    w_i^k-\theta_{k+1}
    \left(
        y_n^k-y_i^{k+1}
    \right),
    \quad i=1,\ldots,n-1,
    \\
    x_n^{k+1}
    &\in
    \operatorname*{argmin}_{x\in\mathbb{R}^{\ell_n}}
    \left\{
        f_n(x)
        +
        \frac{\tau}{2(n-1)}
        \left\|
            A_nx-b+\textstyle\frac1\tau\sum_{i=1}^{n-1}w_i^{k+1}
        \right\|^2
    \right\},
    \\
    y_n^{k+1}
    &=
    \frac1{n-1}\textstyle\sum_{i=1}^{n-1}w_i^{k+1}
    +
    \frac{\tau}{n-1}
    \left(
        A_nx_n^{k+1}-b
    \right).
\end{aligned}
\right.
\end{equation}

To compare \cref{eq:ADT-post-shift} with the method of Alcantara--Dao--Takeda, introduce the auxiliary variables
\begin{equation}
\label{eq:ADT-s-def}
    s_i^k:=y_n^k-w_i^k,\quad \forall i=1,\ldots,n-1.
\end{equation}
Then \(-w_i^k+2y_n^k = y_n^k+s_i^k\) and $w_i^k=y_n^k-s_i^k$, for all $i=1,\ldots,n-1$. Using these identities, after some computations \cref{eq:ADT-post-shift} can be equivalently expressed as
\begin{equation}
\label{eq:ADT-final}
\left\{
\begin{aligned}
    x_i^{k+1}
    &\in
    \operatorname*{argmin}_{x\in\mathbb{R}^{\ell_i}}
    \left\{
        f_i(x)
        +
        \frac{\tau}{2}
        \left\|
            A_ix+\frac{y_n^k+s_i^k}{\tau}
        \right\|^2
    \right\},
    \quad i=1,\ldots,n-1,
    \\
    y_i^{k+1}
    &=
    y_n^k+s_i^k+\tau A_ix_i^{k+1},
    \quad i=1,\ldots,n-1,
    \\
    w_i^{k+1}
    &=
    (1-\theta_{k+1})y_n^k
    +
    \theta_{k+1} y_i^{k+1}
    -
    s_i^k,
    \quad i=1,\ldots,n-1,
    \\
    x_n^{k+1}
    &\in
    \operatorname*{argmin}_{x\in\mathbb{R}^{\ell_n}}
    \left\{
        f_n(x)
        +
        \frac{\tau}{2(n-1)}
        \left\|
            A_nx-b+\frac1\tau
            \sum_{i=1}^{n-1}w_i^{k+1}
        \right\|^2
    \right\},
    \\
    y_n^{k+1}
    &=
    \frac1{n-1}\sum_{i=1}^{n-1}w_i^{k+1}
    +
    \frac{\tau}{n-1}
    \left(
        A_nx_n^{k+1}-b
    \right),
    \\
    s_i^{k+1}
    &=
    s_i^k
    -(1-\theta_{k+1})y_n^k
    +y_n^{k+1}
    -\theta_{k+1} y_i^{k+1},
    \quad i=1,\ldots,n-1.
\end{aligned}
\right.
\end{equation}
Taking $\theta_k=2\kappa$ for all $k\geq1$, with
$\kappa\in(0,1)$, and setting
\[
    \gamma=\delta=\tau,
    \quad
    \lambda=\mu=2,
\]
we have $\theta_k=\kappa\lambda$, while $(\lambda-1)(\mu-1)=1$ and $\delta=\gamma(\lambda-1)$. Thus, up to a change of notation, \cite[Algorithm~2]{alcantaraADMM} coincides with \cref{eq:ADT-final}.
\end{example}

\subsection{A complete-graph multiblock ADMM}
\label{subsec:complete-ADMM}

In the previous subsection we showed that several existing multiblock ADMM schemes can be recovered from particular choices of graphs. We now illustrate how the same framework can be used to design a new method. To this end, we take both graphs to be complete, i.e., every pair of nodes is connected, which yields the densest possible dependency pattern among the blocks. More precisely,
\[G=G'=K_n:=(\mathcal N,\mathcal E), \quad \text{with }\mathcal E=\{(j,i):1\leq j<i\leq n\}.\]
In particular, $d_i=n-1$, for all $i=1,\ldots,n$. We use the onto decomposition $ZZ^*=\Lap(K_n)$ given in \cite[Fact~2.17(iii)]{graph-linear}, where $Z\in\R^{n\times(n-1)}$ is defined componentwise by
\begin{equation}
\label{eq:complete-Z}
    Z_{i,j}
    =
    \begin{cases}
        (n-j)t_j, & \text{if } i=j,\\
        -t_j,     & \text{if } i>j,\\
        0,        & \text{otherwise},
    \end{cases}
    \quad
    \text{with } \;t_j:=\sqrt{\frac{n}{(n-j)(n-j+1)}}.
\end{equation}
To avoid the square roots appearing in \cref{eq:complete-Z}, we rescale the governing variables as
\begin{equation}
\label{eq:complete-rescale}
    \widehat w_j^k:=t_jw_j^k,
    \quad \forall j=1,\ldots,n-1,
\end{equation}
and set $\widehat w_n^k:=0$ for convenience. Using this notation, together with the structure of the complete graph, the auxiliary update in \cref{alg:graphADMM} becomes
\begin{equation}
\label{eq:complete-u}
    u_i^{k+1}
    =
    (n-i)\widehat w_i^k
    -
    \sum_{j=1}^{i-1}\widehat w_j^k
    +
    2\sum_{j=1}^{i-1}y_j^{k+1},
    \quad \forall i=1,\ldots,n.
\end{equation}
On the other hand, multiplying the governing update in \cref{alg:graphADMM} by $t_j$, and using \(t_j^2 = \frac{n}{(n-j)(n-j+1)}\), we obtain
\begin{equation}
\label{eq:complete-w-hat}
    \widehat w_j^{k+1}
    =
    \widehat w_j^k
    -
    \theta_{k+1}
    \frac{n}{(n-j)(n-j+1)}
    \left(
        (n-j)y_j^{k+1}
        -
        \sum_{i=j+1}^ny_i^{k+1}
    \right),
    \quad \forall j=1,\ldots,n-1.
\end{equation}
Therefore, the complete-graph instance of \cref{alg:graphADMM} takes the following explicit form stated in \cref{alg:complete-ADMM}.

\begin{algorithm}[ht!]
    \caption{Graph-based ADMM for the complete graph}
    \label{alg:complete-ADMM}
    \begin{algorithmic}[1]\setlength{\itemsep}{3pt}

        \State \textbf{input:}
        problem data $(F_i,A_i)_{i=1}^n$ and $b$ of \cref{primal},
        a decomposition $b=\sum_{i=1}^n b_i$, a parameter $\tau>0$,
        and a sequence of relaxation parameters
        $(\theta_k)_{k\geq1}\subset{]\epsilon,2-\epsilon[}$ for some
        $\epsilon\in]0,1[$

        \State \textbf{initialize:}
        $\widehat w_1^0,\ldots,\widehat w_{n-1}^0\in\Gi$

        \For{$k=0,1,2,\ldots$}

            \For{$i=1,\ldots,n$}

                \State
                \(
                    u_i^{k+1}
                    =
                    (n-i)\widehat w_i^k
                    -
                    \sum_{j=1}^{i-1}\widehat w_j^k
                    +
                    2\sum_{j=1}^{i-1}y_j^{k+1}
                \)

                \State
                \(
                    x_i^{k+1}
                    \in
                    \left(
                        F_i+\frac{\tau}{n-1}A_i^*A_i
                    \right)^{-1}
                    \left(
                        -\frac1{n-1}
                        A_i^*(u_i^{k+1}-\tau b_i)
                    \right)
                \)
                \algcomment{primal sequence}

                \State
                \(
                    r_i^{k+1}
                    :=
                    A_ix_i^{k+1}-b_i
                \)
                \algcomment{residual sequence}

                \State
                \(
                    y_i^{k+1}
                    :=
                    \frac1{n-1}
                    \left(
                        u_i^{k+1}+\tau r_i^{k+1}
                    \right)
                \)
                \algcomment{dual sequence}

            \EndFor

            \For{$j=1,\ldots,n-1$}

                \State
                \(
                    \widehat w_j^{k+1}
                    =
                    \widehat w_j^k
                    -
                    \theta_{k+1}
                    \frac{n}{(n-j)(n-j+1)}
                    \left(
                        (n-j)y_j^{k+1}
                        -
                        \sum_{i=j+1}^ny_i^{k+1}
                    \right)
                \)
                \algcomment{governing sequence}

            \EndFor

        \EndFor

    \end{algorithmic}
\end{algorithm}

\begin{remark}
Some comments regarding \cref{alg:complete-ADMM} are in order.
\begin{enumerate}[label={\rm(\roman*)}]
\item The auxiliary definition $\widehat w_n^k:=0$ is introduced only for notational convenience, so that the case $i=n$ can be included in \cref{eq:complete-u}. In fact, $\widehat w_n^k$ need not be defined or stored, since it is always multiplied by the factor $n-i$ (which is zero for $i=n$). Thus, the governing sequence still consists of only $n-1$ blocks.

\item Although the sums in \cref{alg:complete-ADMM} are written explicitly for clarity, they can be evaluated recursively with linear cost. Indeed, introducing
\[C_1^{k+1}:=0, \quad\text{and}\quad C_{i+1}^{k+1}:=C_i^{k+1} - \widehat w_i^k + 2y_i^{k+1},\]
the $u_i^{k+1}$-updates reduce to \( u_i^{k+1} =(n-i)\widehat w_i^k+C_i^{k+1}\), for $i=1,\ldots,n$. A similar trick allows the sums in the governing updates to be computed recursively, so that all graph-dependent operations require only a linear number of vector additions and scalar multiplications.

\item The primal, residual, and dual sequences generated by \cref{alg:complete-ADMM} coincide exactly with those of \cref{alg:graphADMM}, while $(\widehat {\bm w}^k)_{k\geq 0}$ is merely a rescaling of its governing sequence $({\bm w}^k)_{k\geq 0}$   given by \cref{eq:complete-rescale}. Therefore, all the convergence results established in \cref{sec:graph-admm} for \cref{alg:graphADMM} apply directly to \cref{alg:complete-ADMM}.
\end{enumerate}
\end{remark}

\section{A (reduced) primal-dual reformulation}
\label{sec:multiplier-only}

Conventionally, classical ADMM is presented as an algorithm involving three variables: two blocks of primal variables $(x_1^k,x_2^k)$ and a single vector $y^k$, usually interpreted as a dual variable due to its characterization as a Lagrange multiplier. By contrast, our dual construction naturally gives rise to an additional governing sequence $\bm w^k$, which is absent in the classical approach.

The explicit expression of the governing sequence in the representation of graph splitting methods is important, not only for unpacking algorithmic duality, but also for applying algorithmic extrapolation schemes that exploit the dynamics of the governing variable \cite{calcan2024admm,lindstrom2022computable,poon2019trajectory}. Nevertheless, the derivation of classical ADMM in \cref{ex:classADMM} showed that this sequence can be eliminated. Moreover, after doing so, one of the two dual shadows becomes redundant, leaving precisely the usual representation in terms of two primal blocks and a single dual variable.

The purpose of this section is to extend this elimination procedure to the general graph setting. More precisely, we eliminate the governing sequence $\bm w^k$ and one dual shadow as independent variables, obtaining an equivalent algorithm expressed in terms of the $n$ primal blocks in $\bm x^k$ and $n-1$ dual shadow blocks, namely the first $n-1$ components of $\bm y^k$. The remaining dual is determined algebraically whenever it is required.

A further consequence of eliminating the governing sequence is that the particular onto decomposition $Z$ is no longer required. Indeed, $Z$ enters the elimination only through $ZZ^*=\operatorname{Lap}(G')$, whose action is completely determined by the graph structure of $G'$. Hence, before stating the primal-dual reformulation in \cref{alg:multiplier-only}, we introduce the following notation associated with the subgraph $G'$. For each $i=1,\ldots,n$, we denote by
\[ \mathcal N_i' := \{h<i:(h,i)\in\mathcal E'\} \cup \{h>i:(i,h)\in\mathcal E'\}\]
the set of neighbors of node $i$ in $G'$, and by \(d_i':=|\mathcal N_i'|\) its degree in this subgraph.

\begin{algorithm}[ht!]
    \caption{Graph-based family of multiblock ADMM (primal--dual form)}
    \label{alg:multiplier-only}
    \begin{algorithmic}[1]\setlength{\itemsep}{3pt}

		\State \textbf{input:}
		problem data $(F_i,A_i)_{i=1}^n$ and $b$ of \cref{primal},
		a decomposition $b=\sum_{i=1}^n b_i$,
		an algorithmic graph $G=(\mathcal N,\mathcal E)$ with degrees
		$(d_i)_{i=1}^n$ and degree balances
		$(\delta_i)_{i=1}^n$,
		a connected spanning subgraph
		$G'=(\mathcal N,\mathcal E')$ with neighbor sets $(\mathcal N_i')_{i=1}^n$
		and  degrees $(d_i')_{i=1}^n$,
		a parameter $\tau>0$, and a sequence of relaxation parameters
		$(\theta_k)_{k\geq0}\subset{]\epsilon,2-\epsilon[}$ for some
		$\epsilon\in]0,1[$

        \State \textbf{initialize:}
        $x_1^0\in\Hi_1,\ldots,x_n^0\in\Hi_n$ and
        $y_1^0,\ldots,y_{n-1}^0\in\Gi$, and set
        \(r_i^0:=A_ix_i^0-b_i,\, i=1,\ldots,n\).

        \For{$k=0,1,2,\ldots$}

            \State
            \(
                y_n^k
                :=
                \frac{\tau}{d_n}\sum_{i=1}^n r_i^k
                +
                \frac1{d_n}
                \sum_{i=1}^{n-1}\delta_i y_i^k
            \)\algcomment{complete dual state}

            \For{$i=1,\ldots,n$}

                \State
				\(
				\begin{aligned}
				    u_i^{k+1}
				    ={}&
				    \underbrace{
				        d_i y_i^k
				        +
				        2\sum_{(h,i)\in\mathcal E}
				        \left(y_h^{k+1}-y_h^k\right)
				    }_{G}
				    -
				        \theta_{k}
				     \underbrace{   \left(
				            d_i' y_i^k
				            -
				            \sum_{h\in\mathcal N_i'}y_h^k
				        \right)
				    }_{G'}
				    -
				    \tau r_i^k.
				\end{aligned}
				\)

                \State
                \(
                    x_i^{k+1}
                    \in
                    \left(
                        F_i+\frac{\tau}{d_i}A_i^*A_i
                    \right)^{-1}
                    \left(
                        -\frac1{d_i}
                        A_i^*(u_i^{k+1}-\tau b_i)
                    \right)
                \)
                \algcomment{primal sequence}

                \State
                \(
                    r_i^{k+1}
                    :=
                    A_ix_i^{k+1}-b_i
                \)
                \algcomment{residual sequence}

                \If{$i<n$}
                    \State
                    \(
                        y_i^{k+1}
                        :=
                        \frac1{d_i}
                        \left(
                            u_i^{k+1}+\tau r_i^{k+1}
                        \right)
                    \)
                    \algcomment{reduced dual sequence}
                \EndIf

            \EndFor

        \EndFor

    \end{algorithmic}
\end{algorithm}

\begin{remark}
Although $y_n^k$ is explicitly displayed in \cref{alg:multiplier-only}, it is not an independent algorithmic variable. Indeed, at every iteration it is completely determined by the primal residuals and the first $n-1$ dual shadows. Thus, the actual state of the method consists only of
\[
    \bm x^k:=(x_1^k,\ldots,x_n^k)\in\bHi
    \quad\text{and}\quad
    \widetilde{\bm y}^k
    :=
    (y_1^k,\ldots,y_{n-1}^k)\in\Gi^{n-1}.
\]
The quantity $y_n^k$ is written explicitly in the algorithm merely for notational convenience, since it allows the updates to be expressed uniformly in terms of the full dual vector $\bm y^k$. Equivalently, one could substitute the expression for $y_n^k$ wherever it appears, thereby obtaining an iteration involving only $(\bm x^k,\widetilde{\bm y}^k)$, at the expense of more cumbersome formulas.
\end{remark}

\begin{theorem}
\label{thm:multiplier-only}
Let $G=(\mathcal N,\mathcal E)$ be an algorithmic graph and $G'=(\mathcal N,\mathcal E')$ a connected spanning subgraph of $G$. Consider \cref{alg:multiplier-only} initialized with $\bm x^0:=(x_1^0,\ldots,x_n^0)\in\bHi$ and $\widetilde{\bm y}^0:=(y_1^0,\ldots,y_{n-1}^0)\in\Gi^{n-1}$. Set \(r_i^0:=A_ix_i^0-b_i\), for each $i=1,\ldots,n$, and let
\begin{equation}\label{eq:pd-yn0}
y_n^0 := \frac{\tau}{d_n}\sum_{i=1}^n r_i^0 +\frac1{d_n}\sum_{i=1}^{n-1}\delta_i y_i^0.
\end{equation}
Let $Z\in\R^{n\times(n-1)}$ with $ZZ^*=\operatorname{Lap}(G')$ be any onto decomposition of the Laplacian of $G'$. Then there exists a unique $\bm w^0\in\Gi^{n-1}$ satisfying
\begin{equation}
\label{eq:pd-compatible-initialization}
\begin{aligned}
    (Z\bm w^0)_i
    ={}&
    d_i y_i^0
    -
    2\sum_{(h,i)\in\mathcal E}y_h^0
    -
    \theta_0
    \left(
        d_i'y_i^0
        -
        \sum_{h\in\mathcal N_i'}y_h^0
    \right)
    -
    \tau r_i^0,
    \quad \forall i=1,\ldots,n.
\end{aligned}
\end{equation}
If \cref{alg:graphADMM}, executed with graph setting $(G,G',Z)$, is initialized from this $\bm w^0$, then both algorithms generate the same primal $(\bm x^k)_{k\geq1}$ and dual sequences $(\bm y^k)_{k\geq1}$.
\end{theorem}

\begin{proof}
We first show the existence of $\bm w^0\in\Gi^{n-1}$ satisfying \cref{eq:pd-compatible-initialization}. To this end, define $\bm q^0\in\Gi^n$ componentwise by
\[
    q_i^0
    :=
    d_i y_i^0
    -
    2\sum_{(h,i)\in\mathcal E}y_h^0
    -
    \theta_0
    \left(
        d_i'y_i^0
        -
        \sum_{h\in\mathcal N_i'}y_h^0
    \right)
    -
    \tau r_i^0,\quad \forall i=1,\ldots,n.
\]
By counting the contributions induced by each edge of $G'$ and $G$, respectively, we have
\begin{subequations}
\begin{align}
    & \textstyle\sum_{i=1}^n
    \left(
        d_i'y_i^k
        -
        \sum_{h\in\mathcal N_i'}y_h^k
    \right)
    =0,\\
    &\textstyle\sum_{i=1}^n
    \sum_{(h,i)\in\mathcal E}y_h^k
    =
    \sum_{i=1}^n d_i^{\rm out}y_i^k,\label{eq:pd-aux_dout}
\end{align}
\end{subequations}
for all $k\geq 0$. In particular,
\begin{equation}\label{eq:pd-sumq0}
    \sum_{i=1}^n q_i^0
    =
    \sum_{i=1}^n
    \left(d_i-2d_i^{\rm out}\right)y_i^0
    -
    \tau\sum_{i=1}^n r_i^0
    =
    -\sum_{i=1}^n\delta_i y_i^0
    -
    \tau\sum_{i=1}^n r_i^0.
\end{equation}
Since $G$ is algorithmic, $\delta_n=d_n^{\rm out}-d_n^{\rm in}=0-d_n=-d_n$. Therefore, from \cref{eq:pd-yn0} we get
\[ \sum_{i=1}^n\delta_i y_i^0 = -\tau\sum_{i=1}^n r_i^0,\]
which combined with \cref{eq:pd-sumq0} gives $\sum_{i=1}^n q_i^0=0$. Then, by \cite[Proposition 2.18]{graph-linear}, there exists $\bm w^0\in\Gi^{n-1}$ such that \( Z\bm w^0=\bm q^0\). Furthermore, $Z$ is injective, so $\bm w^0$ is unique.

We next derive the relation between the governing sequence and the primal--dual variables. For every $k\geq1$, combining the auxiliary and the dual updates in steps 5 and 8 of \cref{alg:graphADMM}, respectively, we obtain
\begin{equation}\label{eq:pd-auxdual-relation}
    (Z\bm w^{k-1})_i
    =
    d_i y_i^k
    -
    2\sum_{(h,i)\in\mathcal E}y_h^k
    -
    \tau r_i^k,\quad \forall i=1,\ldots,n.
\end{equation}
On the other hand, the governing update reads as \(\bm w^k = \bm w^{k-1} - \theta_k Z^*\bm y^k\). Multiplying by $Z$ and using $ZZ^*=\operatorname{Lap}(G')$ and \cref{eq:pd-auxdual-relation}, we obtain
\begin{equation}
\label{eq:pd-governing-relation}
\begin{aligned}
    (Z\bm w^k)_i
    ={}&
    d_i y_i^k
    -
    2\sum_{(h,i)\in\mathcal E}y_h^k
    -
    \tau r_i^k
    -
    \theta_k
    \big(\operatorname{Lap}(G')\bm y^k\big)_i
    \\
    ={}&
    d_i y_i^k
    -
    2\sum_{(h,i)\in\mathcal E}y_h^k
    -
    \theta_k
    \left(
        d_i'y_i^k
        -
        \sum_{h\in\mathcal N_i'}y_h^k
    \right)
    -
    \tau r_i^k,
\end{aligned}
\end{equation}
for all $k\geq1$ and $i=1,\ldots,n$. In fact, by \cref{eq:pd-compatible-initialization}, the same relation holds for $k=0$.

We can now compare the updates of the two algorithms. Using \cref{eq:pd-governing-relation}, the auxiliary update in \cref{alg:multiplier-only} becomes
\[
\begin{aligned}
    u_i^{k+1}
    ={}&    
        d_i y_i^k
        +
        2\sum_{(h,i)\in\mathcal E}
        \left(y_h^{k+1}-y_h^k\right)
        -
        \theta_k
        \left(
            d_i'y_i^k
            -
            \sum_{h\in\mathcal N_i'}y_h^k
        \right)
    -
    \tau r_i^k
    =
    (Z\bm w^k)_i
    +
    2\sum_{(h,i)\in\mathcal E}y_h^{k+1},
\end{aligned}
\]
which is precisely the auxiliary update of \cref{alg:graphADMM}. Since the nodes are processed in the same order in both algorithms and every $(h,i)\in\mathcal E$ satisfies $h<i$, the corresponding primal, residual, and first $n-1$ dual updates therefore coincide.

It remains to verify that the value of the last dual shadow recovered in step 4 of \cref{alg:multiplier-only} coincides with the one generated by \cref{alg:graphADMM}. Rearranging \cref{eq:pd-auxdual-relation} we have
\[
    d_i y_i^{k+1}
    =
    (Z\bm w^k)_i
    +
    2\sum_{(h,i)\in\mathcal E}y_h^{k+1}
    +
    \tau r_i^{k+1},\quad \forall i=1,\ldots,n.
\]
Summing these identities and using $\bm1_n^*Z=0$ and \cref{eq:pd-aux_dout}, we obtain
\[
    \sum_{i=1}^n
    \left(d_i-2d_i^{\rm out}\right)y_i^{k+1}
    =
    \tau\sum_{i=1}^n r_i^{k+1}.
\]
Equivalently, \(\sum_{i=1}^n\delta_i y_i^{k+1} = -\tau\sum_{i=1}^n r_i^{k+1}\). Since $\delta_n=-d_n$, this yields
\[
    y_n^{k+1}
    =
    \frac{\tau}{d_n}
    \sum_{i=1}^n r_i^{k+1}
    +
    \frac1{d_n}
    \sum_{i=1}^{n-1}\delta_i y_i^{k+1},
\]
which is exactly the recovery formula in \cref{alg:multiplier-only}.

In summary, starting from the compatible initialization \cref{eq:pd-compatible-initialization}, all the updates of the two algorithms coincide, and hence they generate the same primal and dual sequences.
\end{proof}

Besides providing a primal--dual representation of the graph-based
ADMM family, this formulation clarifies its connection with classical
ADMM (see \cref{rem:classical-multiplier-only}). This form for classical ADMM is often preferable because the role of the shadow variable as a Lagrange multiplier is apparent from the augmented Lagrangian expression. In the general
multiblock setting, however, the interpretation of the individual shadow
variables as Lagrange multipliers is less immediate.

\begin{remark}[The classical ADMM for $n=2$ revisited]
\label{rem:classical-multiplier-only}

The construction of \cref{alg:multiplier-only} can be viewed as a
generalization of the elimination procedure used to recover the
classical ADMM in \cref{ex:classADMM}. Indeed, the unique
algorithmic graph of order two satisfies
\[
    d_1=d_2=d_1'=d_2'=1,
    \quad
    \delta_1=1
    \quad\text{and}\quad
    \delta_2=-1.
\]
Take $\theta_k\equiv 1$ and recall that we apply the algorithms to the primal data in reverse order and retain the original block indices; that is, the first graph node corresponds
to block $(f_2,A_2,b)$ and the second one to $(f_1,A_1,0)$. Hence, the only dual
shadow retained by \cref{alg:multiplier-only} is $y_2^k$.

The recovery formula for the eliminated dual gives
\(
    y_1^k
    =
    y_2^k
    +
    \tau
    \left(
        A_1x_1^k+A_2x_2^k-b
    \right).
\)
Consequently, $u_2^{k+1}=y_2^k+\tau A_1x_1^k$ and
\[
    x_2^{k+1}
    \in
    \operatorname*{argmin}_{x_2\in\Hi_2}
    \left\{
        f_2(x_2)
        +
        \frac{\tau}{2}
        \left\|
            A_1x_1^k+A_2x_2-b+\frac{y_2^k}{\tau}
        \right\|^2
    \right\}.
\]
Moreover,
\(
    y_2^{k+1}
    =
    y_2^k
    +
    \tau
    \left(
        A_1x_1^k+A_2x_2^{k+1}-b
    \right),
\)
so for the second block
\(
    u_1^{k+1}
    =
    y_2^k
    +
    2(y_2^{k+1}-y_2^k)
    -
    \tau A_1x_1^k=
    y_2^{k+1}
    +
    \tau(A_2x_2^{k+1}-b),
\)
which yields
\[
    x_1^{k+1}
    \in
    \operatorname*{argmin}_{x_1\in\Hi_1}
    \left\{
        f_1(x_1)
        +
        \frac{\tau}{2}
        \left\|
            A_1x_1+A_2x_2^{k+1}-b
            +\frac{y_2^{k+1}}{\tau}
        \right\|^2
    \right\}.
\]
Thus, \cref{alg:multiplier-only} reproduces exactly the two elimination steps already used in \cref{ex:classADMM} to obtain \cref{eq:classical-admm-cyclic}: first, the governing sequence is removed; second, one dual shadow ceases to be an independent variable. Hence the iteration reduces to the two primal blocks and a single dual shadow, namely $y_2^k$, which is precisely the classical ADMM multiplier.
\end{remark}

\section{Numerical experiments}
\label{sec:numerics}

In this section, we numerically compare the behavior of the
graph-based ADMM schemes for different graph configurations.
All algorithms were implemented in Python~3.9 and run on a Mac mini with Apple M4 CPU, with 24 GB RAM, under macOS Tahoe 26.6.2. The code for these experiments is available at \url{https://github.com/OPT-UA/Graph-ADMM}.

\subsection{Common settings}\label{subsec:common}

\def\n{6}
\def\m{100}

\paragraph{The problem}
As in \cref{sec:convex-minimization}, the ambient spaces are finite-dimensional, with $\Hi_i=\R^{\ell_i}$ for $i=1,\ldots,n$ and $\Gi=\R^m$. Throughout, we set $n=\n$ and $m=\m$. We consider linearly constrained quadratic problems of the form
\begin{equation}\tag{PE1}\label{experiment1}
    \min_{x_1\in\R^{\ell_1},\ldots,x_\n\in\R^{\ell_\n}}
    \sum_{i=1}^\n \frac12\|Q_ix_i\|^2
    \quad\text{subject to}\quad
    \sum_{i=1}^\n A_ix_i=b,
\end{equation}
where $Q_i\in\R^{p_i\times\ell_i}$ and $A_i\in\R^{\m\times\ell_i}$, for all $i=1,\ldots,\n$.

We concern on quadratic problems as they provide a flexible setting in which to compare different graph-based schemes under a common problem structure. By varying the dimensions and the matrices $Q_i$ and $A_i$, one can generate a broad range of test instances while retaining a simple and uniform implementation. In particular, each primal update $x_i^{k+1}$ reduces to solving the linear system
\begin{equation}
\label{eq:quadratic-primal-update}
    \left(d_iQ_i^*Q_i+\tau A_i^*A_i\right)x_i^{k+1}
    =
    A_i^*(\tau b_i-u_i^{k+1}).
\end{equation}
Given the efficiency of current solvers, this ensures that differences in performance can be attributed to the graph topology.

\paragraph{Parameter setting}
For each problem, we randomly choose $\ell_i\sim U\{1,\dots,10\}$ for all $i=1,\ldots,\n$. Given these dimensions, we then randomly select $p_i\sim U\{\ell_i,\ldots,10\}$ and the matrices $A_i\sim U[0,1]^{\m\times\ell_i}$ and $Q_i\sim U[0,1]^{p_i\times\ell_i}$ for all $i=1,\ldots,\n$.

The test problems are generated so that the assumptions of
\cref{cor:convergence-convex} are automatically satisfied.
More precisely, every matrix $Q_i$ is chosen to have full column
rank so that strong convexity holds and \cref{B2}\ref{B2b} holds. Moreover, since quadratic functions have full domain, the relative-interior condition in \cref{B1}\ref{B1b} reduces to feasibility. To guarantee this, for each instance we first
generate a point
$\widehat{\bm x}=(\widehat x_1,\ldots,\widehat x_n)$, where $\widehat x_i\sim U[0,1]^{\ell_i}$,
and then define
\[
    b:=\sum_{i=1}^n A_i\widehat x_i.
\]
Finally, we choose the decomposition $b_i=b/\n$ for all $i=1,\ldots,\n$. Therefore, all the assumptions of
\cref{cor:convergence-convex} (including the one required for the convergence of the primal sequence) are satisfied. Furthermore, since the
objective is strongly convex, \cref{experiment1} has a unique
solution, and so does the linear system \cref{eq:quadratic-primal-update}.

\paragraph{Stopping criterion}
To define an adequate stopping criterion, we use the primal--dual optimality
conditions associated with \cref{experiment1}. A primal--dual
solution $(\bm x^\star,y^\star)$ satisfies
\begin{equation}
\label{eq:quadratic-KKT}
    \sum_{i=1}^n A_ix_i^\star-b=0,
    \quad\text{and}\quad
    Q_i^*Q_ix_i^\star+A_i^*y^\star=0,
    \quad \forall i=1,\ldots,\n.
\end{equation}
By the optimality condition of the primal subproblem (step~6 of \cref{alg:graphADMM-convex}), we have
\[
    Q_i^*Q_ix_i^{k+1}+A_i^*y_i^{k+1}=0,
    \quad \forall i=1,\ldots,n,
\]
at every iteration. Thus, each block already satisfies its
stationarity condition with respect to its corresponding dual variable. Convergence to a primal--dual solution therefore requires
primal feasibility together with dual consensus. Accordingly, we define the primal and dual
residuals as
\begin{subequations}
\begin{align}
    R_{\mathrm p}^{k+1}
    &:=
    \left\|
        \textstyle\sum_{i=1}^nA_ix_i^{k+1}-b
    \right\|,
    \label{eq:numerical-primal-residual}\\
    R_{\mathrm d}^{k+1}
    &:=
    \left\|
        \bm y^{k+1}
        -\bm1_n\overline y^{k+1}
    \right\|,
    \label{eq:numerical-dual-residual}
\end{align}
\end{subequations}
where $\bm y^{k+1}:=(y_1^{k+1},\ldots,y_n^{k+1})$ and $\overline y^{k+1}:=\frac{1}{n}\sum_{i=1}^n y_i^{k+1}$.
Under the assumptions imposed above,
\cref{thm:convergence}\ref{thm:convergence_r} and
\cref{lem:equiv}\ref{lem:equiv_ybar} yield
\(\left(R_{\mathrm p}^{k}\right)^2=o(1/k)\) and  \(\left(R_{\mathrm d}^{k}\right)^2=o(1/k)\),
so both residuals vanish asymptotically. We therefore terminate the
iterations whenever
\begin{equation}
\label{eq:numerical-stopping}
    \max\left\{
        R_{\mathrm p}^{k+1},
        R_{\mathrm d}^{k+1}
    \right\}
    <\epsilon,
\end{equation}
where $\epsilon>0$ is a tolerance specified in each context.

\paragraph{Considered algorithms}
We compare six graph configurations. Since all the graphs involved
have the same node set $\mathcal N=\{1,\ldots,\n\}$, it suffices to
specify their edge sets. The pairs $(\mathcal E,\mathcal E')$
defining the methods under comparison are listed in
\cref{tab:graph-methods}.

\begin{table}[H]
    \centering\renewcommand{\arraystretch}{1.2}
    \begin{tabular}{c|c|c}
        Method & $\mathcal E$ & $\mathcal E'$ \\
        \hline
        Malitsky--Tam
        &
        $\{(i,i+1)\}_{i=1}^{n-1}\cup\{(1,n)\}$
        &
        $\{(i,i+1)\}_{i=1}^{n-1}$
        \\
        Sequential
        &
        $\{(i,i+1)\}_{i=1}^{n-1}$
        &
        $\{(i,i+1)\}_{i=1}^{n-1}$
        \\
        Alcantara--Dao--Takeda
        &
        $\{(1,i)\}_{i=2}^{n}$
        &
        $\{(1,i)\}_{i=2}^{n}$
        \\
        Parallel Down
        &
        $\{(i,n)\}_{i=1}^{n-1}$
        &
        $\{(i,n)\}_{i=1}^{n-1}$
        \\
        Complete
        &
        $\{(j,i)\}_{1\leq j<i\leq n}$
        &
        $\{(j,i)\}_{1\leq j<i\leq n}$
        \\
        Generalized Ryu
        &
        $\{(j,i)\}_{1\leq j<i\leq n}$
        &
        $\{(i,n)\}_{i=1}^{n-1}$
        \\
    \end{tabular}
    \caption{Edge-set pairs $(\mathcal E,\mathcal E')$ defining the
    graph configurations of the multiblock ADMM algorithms compared
    in the numerical experiment}
    \label{tab:graph-methods}
\end{table}

\subsection{Choosing optimal stepsize and relaxation}

Our main goal is to compare the performance of \cref{alg:graphADMM-convex} across the selected graph configurations. However, their optimal performance might be achieved at different choices of the stepsize~$\tau$ and relaxation parameter $\theta$. Thus, in order to ensure fairness, we first tune these parameters for each graph setting.

To approximately find the best-performing parameters, we tested every pair $(\tau,\theta)$ in the grid $\{0.4,0.5,\ldots,4\}\times \{1,1.05,\ldots,1.95\}$. Although $\theta$ is theoretically bounded in $(0,2)$, preliminary experiments showed that values of $\theta<1$ were dominated. Thus, each graph configuration was tested over $37\times 20=740$ pairs $(\tau,\theta)$. We generated five problem instances (as described in \cref{subsec:common}) and five independent random initial points, whose components were sampled from $N(0,1)$. Thus, the same $5\times5$ combinations of problems and initial points were tested for every pair of parameters and every method. For this experiment, we set $\epsilon=10^{-4}$ in \cref{eq:numerical-stopping}.

\paragraph{Results}
For each method and pair $(\tau,\theta)$, let $k_{\tau,\theta}$ denote the average number of iterations required to satisfy \eqref{eq:numerical-stopping} among all problems and initializations. The resulting values are displayed in \cref{fig:optimal-quadratics} as a heat map for each method. The $x$-axis represents the stepsize $\tau$ and the $y$-axis the relaxation parameter $\theta$, while the color scale represents the corresponding value of $k_{\tau,\theta}$. The point attaining the smallest value of $k_{\tau,\theta}$ is marked by a red dot.

Regarding the relaxation parameter~$\theta$, we confirm from the results that overrelaxed algorithms are preferable. Those algorithms with $G=G'$ have roughly the same relaxation parameter, between $1.75$ and $1.8$, whereas Malitsky--Tam and Generalized Ryu behave better with values closer to the upper bound. In contrast, the best-performing stepsize $\tau$ is much more variable, highlighting the importance of tuning $\tau$ separately for each method. 

We also observe that Complete achieved the lowest number of iterations at the optimal parameters. In addition, it permits more flexibility in the choice of the pair $(\tau,\theta)$, as it exhibits a wider region of low values of $k_{\tau,\theta}$. Finally, although taking extremely small stepsizes slows down the convergence of all algorithms, this is specially emphasized for Generalized Ryu.

\begin{figure}[h]
    \centering
    \includegraphics[width=\textwidth]{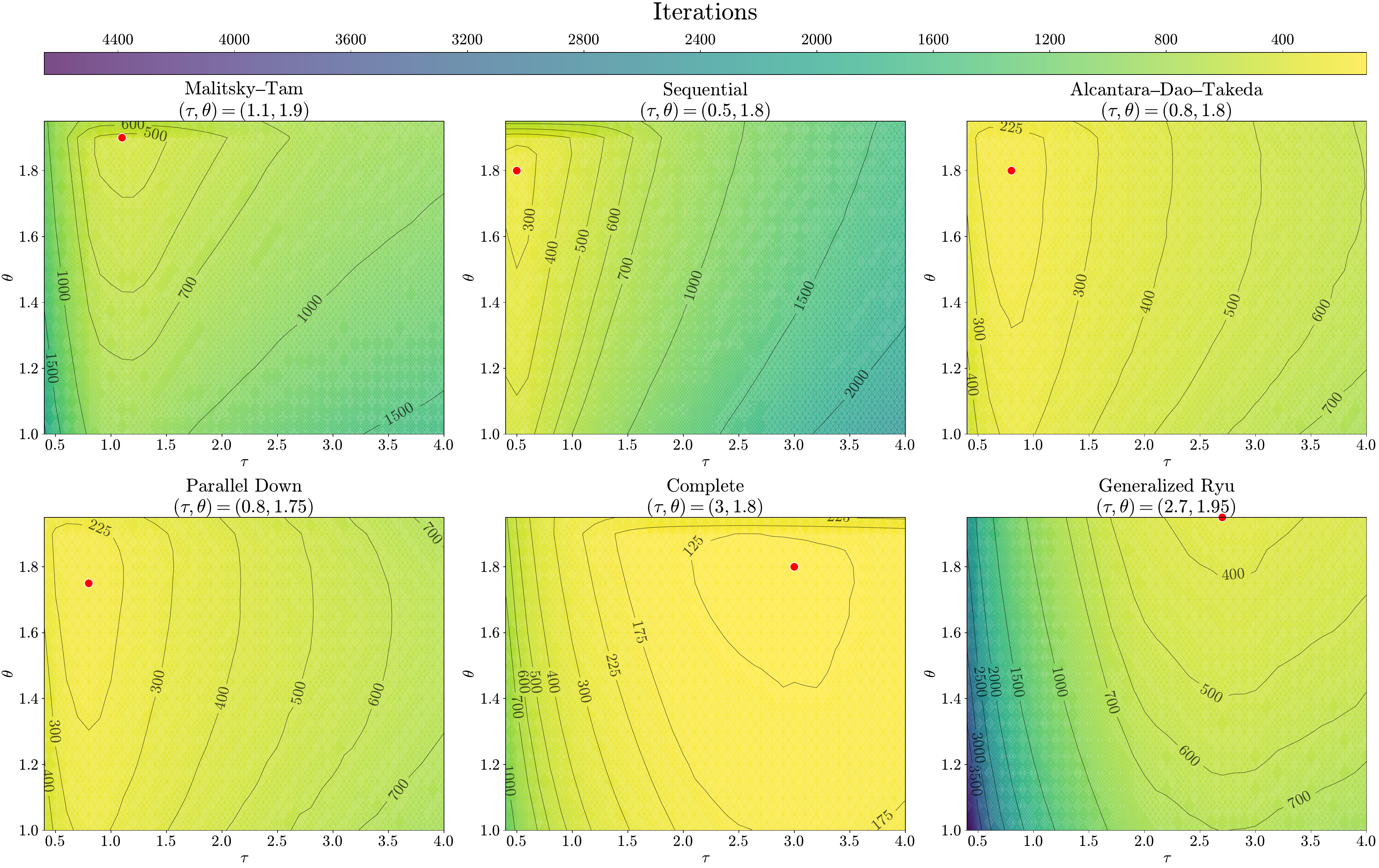}
    \caption{Heat maps of the average number of iterations until the stopping criterion~\cref{eq:numerical-stopping} with $\epsilon=10^{-4}$ is satisfied}
    \label{fig:optimal-quadratics}
\end{figure}

\subsection{Comparing performances}

Now, we are ready to fairly compare the performance of the methods depending on the graph topology. For this experiment, we generated $N=200$ problems following the procedure described in \cref{subsec:common}. Each problem was solved by each of the six methods listed in \cref{tab:graph-methods}, all of them initialized from the starting point $\bm w^0=0$. For each method, the pair of parameters $(\tau,\theta)$ was chosen according to the results in \cref{fig:optimal-quadratics}. The tolerance for the stopping criterion was set to~$\epsilon=10^{-6}$ .

\paragraph{Results} To visualize the results we rely on \emph{performance profiles}, introduced by Dolan and Mor\'e~\cite{performanceprofile}, which are described as follows. For every problem $P$ and method $S$, we record the CPU time $t_{P,S}$ required to satisfy the stopping criterion in \cref{eq:numerical-stopping}. For every problem $P$, we define $t_P^{\min}:=\min_S t_{P,S}$. Then we compute
\begin{equation}\label{eq:performance-profile}
\pi_S(\alpha):=\frac{1}{N}\left|\left\{P:\ t_{P,S}\leq\alpha t_P^{\min}\right\}\right|,\quad \forall \alpha\geq 1.
\end{equation}
Thus, $\pi_S(1)$ is the proportion of problems for which method $S$ is the fastest, while $\pi_S(\alpha)$ gives the proportion of problems for which its running time is within a factor $\alpha$ of the best one. The resulting performance profiles are depicted in \cref{fig:quadratics}.

\begin{figure}[h]
    \centering
    \includegraphics[width=\textwidth]{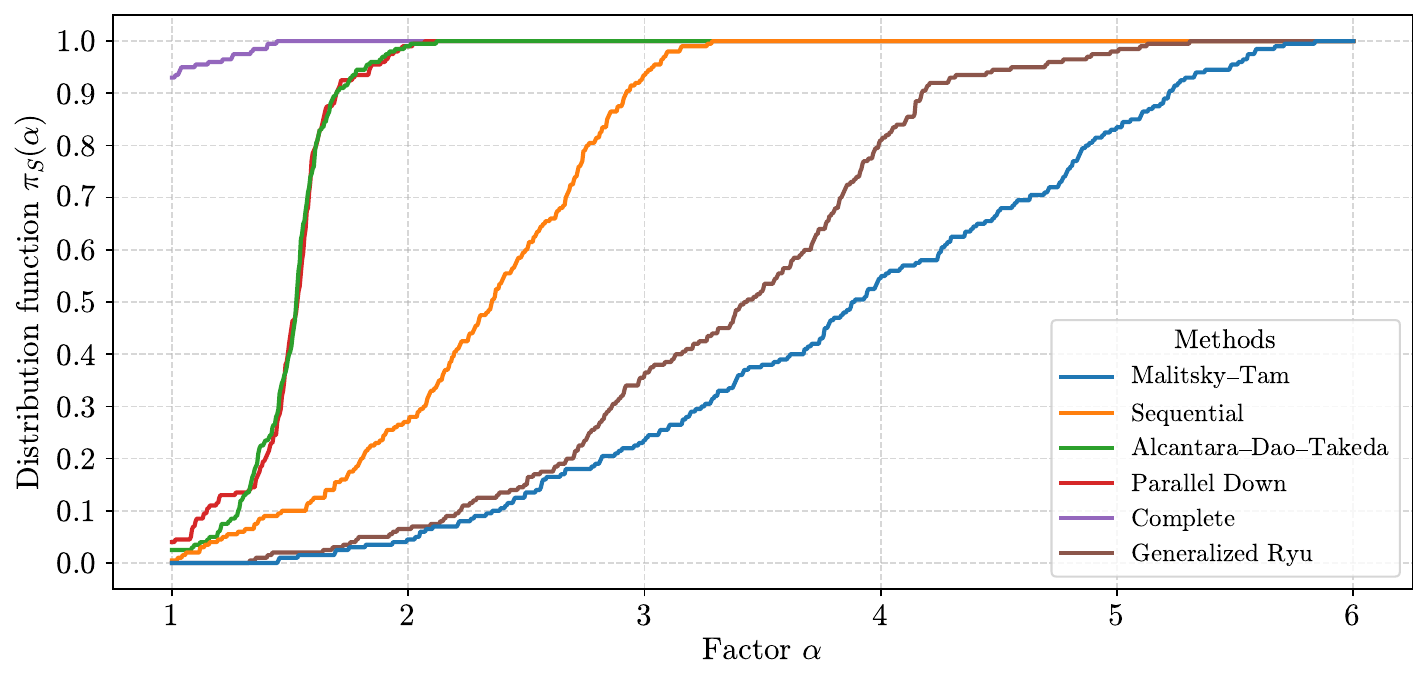}
    \caption{Performance profiles based on CPU time for the quadratic
    problems \cref{experiment1}}
    \label{fig:quadratics}
\end{figure}

We notice the clear superiority of the Complete algorithm, being the fastest method in more than ninety percent of the problems. Next in the ranking are Parallel Down and Alcantara--Dao--Takeda algorithms, whose performance almost coincide, followed by Sequential. This empirically suggests that algorithms with $G=G'$ seem to perform better. Indeed, Generalized Ryu and Malitsky--Tam are the slowest methods, having in common that their graph and subgraph differ.

\section{Conclusion}
\label{sec:conclu}

In this paper, we have introduced a unified family of multiblock ADMM schemes by combining the graph-based splitting framework of \cite{graph-drs} with the complete Moreau identity of \cite{gabayduality}. The resulting methods are determined by a pair of graphs $(G,G')$, whose different topologies induce different coordination patterns among the primal blocks. This framework recovers, as particular instances, several algorithms already proposed in the literature, while also giving rise to new schemes. We have studied in particular the complete-graph configuration, for which we derived an explicit closed-form representation of the resulting multiblock ADMM.

We further showed that the resulting family of algorithms admits an equivalent reformulation in which the governing sequence is eliminated, and the iteration is expressed solely in terms of the primal blocks and a reduced set of dual shadow variables. Besides providing an alternative and more familiar representation of the algorithm, this reformulation clarifies its connection to the classical ADMM, as the Lagrange multiplier appearing in the latter is identified precisely with the single dual shadow variable retained by this primal-dual formulation.

Finally, numerical experiments were conducted to measure and compare how fast the algorithms converge to a solution of the problem. To ensure a fair comparison, we first carried out a grid search to determine the best-performing stepsize and relaxation parameter for each algorithm, so that each method performs optimally. We then based our comparison on performance profiles, which provide clear evidence of the superior performance of the Complete algorithm on the problems considered.

\makeacknowledgements

\printbibliography

\end{document}